\documentclass[10pt,reqno]{amsart}

\usepackage{amsaddr}
\usepackage{amssymb,amsmath,amsthm,amsxtra,amsfonts}
\usepackage{setspace}

\usepackage{caption}

\usepackage{graphicx}
\usepackage{hyperref}

\newtheorem{theorem}{Theorem}[section]

\newtheorem{lemma}[theorem]{Lemma}
\newtheorem{proposition}[theorem]{Proposition}
\numberwithin{equation}{section}
\newtheorem{definition}[theorem]{Definition}

\def\tempsepline{{}-\!\!\!-\!\!\!-\!\!\!-\!\!\!-\!\!\!-\!\!\!-}
\def\tempsep{\smallskip\centerline{$\tempsepline\!\!\!\!\hbox to 0mm{\hss//\hss}\!\tempsepline$}\smallskip}

\def\R{\mathbb R}
\def\D{\mathbb D}

\def\C{\mathbb C}

\DeclareMathOperator{\diag}{diag}
\theoremstyle{remark}
\newtheorem*{remarks}{Remarks}
\newtheorem*{remark}{Remark}

\newcommand{\bbC}{{\mathbb{C}}}
\newcommand{\bbD}{{\mathbb{D}}}
\newcommand{\bbE}{{\mathbb{E}}}

\newcommand{\bbH}{{\mathbb{H}}}

\newcommand{\bbR}{{\mathbb{R}}}

\newcommand{\bbU}{{\mathbb{U}}}
\newcommand{\bbZ}{{\mathbb{Z}}}
\newcommand{\bbO}{{\mathbb{O}}}
\newcommand{\bbSO}{{\mathbb{SO}}}
\newcommand{\bbUSp}{{\mathbb{U}Sp}}

\newcommand{\calC}{{\mathcal{C}}}

\newcommand{\calL}{{\mathcal L}}
\newcommand{\calM}{{\mathcal M}}
\newcommand{\calN}{{\mathcal N}}

\newcommand{\calS}{{\mathcal S}}

\newcommand{\beq}{\begin{equation}}
\newcommand{\eeq}{\end{equation}}
\newcommand{\ba}{\begin{align*}}
\newcommand{\ea}{\end{align*}}

\newcommand{\HR}{{\mathbb{H_R}}}
\newcommand{\bbHR}{{\mathbb{H_R}}}

\newcommand{\bbHC}{{\mathbb{H_C}}}
\newcommand{\ii}{{\mathsf{i}}}
\newcommand{\jj}{{\mathsf{j}}}
\newcommand{\kk}{{\mathsf{k}}}
\newcommand{\qtq}[1]{\quad\text{#1}\quad}
\newcommand{\fC}{{\mathfrak{C}}}

\DeclareMathOperator{\real}{Re}
\DeclareMathOperator{\imag}{Im}

\begin{document}

\title[Truncations of unitary ensembles]{Matrix models and eigenvalue statistics for truncations of classical ensembles
of random unitary matrices}
\author{Rowan Killip and Rostyslav Kozhan}
\address{
         UCLA Mathematics Department\\
         Box 951555, Los Angeles, CA 90095, US}
\address{
         Uppsala University\\
         Box 480, 751 06 Uppsala, Sweden}
\email{killip@math.ucla.edu, rostyslav.kozhan@math.uu.se}







\begin{abstract}
We consider random non-normal matrices constructed by removing one row and column from samples from Dyson's circular ensembles or samples from the classical compact groups.
We develop sparse matrix models whose spectral measures match these ensembles. This allows us to compute the joint law of the eigenvalues, which have a natural interpretation as resonances for open quantum systems or as electrostatic charges located in a dielectric medium.

Our methods allow us to consider all values of $\beta>0$, not merely $\beta=1,2,4$.
\end{abstract}

\maketitle

\section{Introduction}\label{sIntro}

The main objects of investigation of this paper are Dyson's circular ensembles of unitary random matrices, as well as the orthogonal and compact symplectic groups equipped with the Haar measure.  To give the flavor of the results, let us restrict our attention in this introduction to Dyson's circular ensembles only.

In 1962 Dyson~\cite{Dyson} introduced three ensembles of unitary random matrices ($CUE$, $COE$, and $CSE$) to model complex physical systems (e.g., evolution operators of closed quantum systems) corresponding to various physical symmetry classes. We begin with the definitions.  Note that the $CSE$ ensemble is best explained through the use of (real and complex) quaternions.  We review this material and introduce the relevant notations in Appendix~\ref{sQuaternions}.

\begin{definition}\label{DysonEnsembles}
Dyson's three circular ensembles are:
\begin{itemize}
\item[(CUE)] The circular unitary ensemble $CUE(n)$ is the set $($group$)$ of all $n\times n$ unitary matrices $\bbU(n)$ endowed with the Haar measure $($which comes from the group structure of $\bbU(n))$.
\item[(COE)] The circular orthogonal ensemble $COE(n)$ is the set of all $n\times n$  symmetric unitary matrices with the measure induced by the Haar measure on $\bbU(n)$ via the mapping 
    \begin{align*}
    \bbU(n) & \to COE(n) \\
    U & \mapsto U^T U.
    \end{align*}
\item[(CSE)] The circular symplectic ensemble $CSE(n)$ is the set of all $n\times n$ complex quaternionic matrices that are both unitary and self-dual. The measure on this set is taken to be the one induced by the Haar measure on $\bbU(2n)$ via the mapping
\begin{align*}
    \bbU(2n) & \to CSE(n) \\
    U & \mapsto \fC^{-1}(U^R U).
\end{align*}
\end{itemize}
\end{definition}
As argued by Dyson, $COE$ is relevant in most practical circumstances, with the two exceptions of systems without time-reversal invariance (when $CUE$ should be used) and odd-spin systems with time-reversal but without rotational symmetry (when $CSE$ should be used).

Dyson computed that the eigenvalues of these ensembles are jointly distributed on $\partial\bbD^n := \{e^{i\theta}: 0\le \theta <2\pi\}^n$ proportionally to
\begin{equation*}
\propto \prod_{j<k} \left| e^{i \theta_j} - e^{i\theta_k} \right| ^\beta d\theta_1 \ldots d\theta_n,
\end{equation*}
where $\beta=1,2,4$ for $COE(n)$, $CUE(n)$, $CSE(n)$, respectively (for $CSE(n)$ each of the $n$ eigenvalues is of multiplicity $2$).

For each of the listed ensembles we consider their truncations: choose at random a matrix from the ensemble and delete any row and the corresponding column (quarternionic row/column in the CSE case). These truncations appear naturally as the evolution operators of \textit{open} quantum systems, and the eigenvalues of these truncations can be thought of as the scattering resonances of a random quantum system attached to a ``lead''. This and other physical applications of truncations have recently gained a significant popularity in the mathematical and  physical literature. We refer the interested reader to excellent review papers~\cite{FyoSav,FyoSom,KhoSom} and references therein.

The unitary truncations have connections in combinatorics (see~\cite{Nov07}) and to the theory random Gaussian analytic functions (see~\cite{Kri09,For10,Hough}).

To be able to systematically analyze the truncations, it is natural to start with finding the joint eigenvalue distribution. For the truncations of $CUE(n+1)$ it has been computed  in the physics literature by {\.Z}yczkowski--Sommers~\cite{ZSommers} 
(see also~\cite{Petz,ForKri,Nov07}).
 The result is that the eigenvalues are distributed in $\bbD^{n} = \{z\in\bbC: |z|< 1\}^{n}$ according to
\begin{equation*}\label{unitary_2}
    \tfrac{1}{\pi^{n}} \prod_{j,k=1}^{n}   \prod_{j<k} |z_k-z_j|^2 \, d^2 z_1\ldots d^2 z_{n},
\end{equation*}
where $d^2$ is the two-dimensional Lebesgue measure on $\bbD$.
In this paper we solve this problem for $COE$ and $CSE$.

\begin{theorem}\label{introThm}

The eigenvalues $z_1,\ldots,z_{n}$ of $COE(n+1)$, $CUE(n+1)$, $CSE(n+1)$ ensembles with one row and column removed are distributed in $\bbD^n$ according to
\begin{equation}\label{eigenvaluesDyson}
\tfrac{\beta^n}{(2\pi)^n} \prod_{j,k=1}^n (1-z_j \bar{z}_k)^{\frac{\beta}{2}-1}  \prod_{j<k} |z_k-z_j|^2 \, d^2 z_1\ldots d^2 z_n
\end{equation}
with $\beta=1,2,4$ respectively. 
\end{theorem}
\begin{remarks}
1. For $CSE$ each of the eigenvalues is of multiplicity $2$.

2. In fact, we show that for any $0<\beta<\infty$~\eqref{eigenvaluesDyson} is the eigenvalue distribution of the circular $\beta$-ensemble of Killip--Nenciu~\cite{KN} (see Remark 1 after Proposition~\ref{thmDysonCMV}) with the first row and column removed.

3. $\beta=2$ is quite special here as the process becomes determinantal~\cite{ZSommers}.
\end{remarks}

One of the ingredients of the proof is to reduce these truncations to a special sparse matrix form (namely, CMV form, see Definition~\ref{CMV}) that depends on only $n$ independent coefficients with explicit distributions. This is done in Proposition~\ref{thmTruncModels} and could be of interest on its own for simulation purposes.

Another way to look at this is that we show that the eigenvalues of the truncations coincide with the zeros of orthogonal polynomials on the unit circle with random recursion coefficients that are independently distributed according to the explicit distributions~\eqref{truncDyson}.

\begin{figure}
\centering
\minipage{0.32\textwidth}
\includegraphics[width=\linewidth]{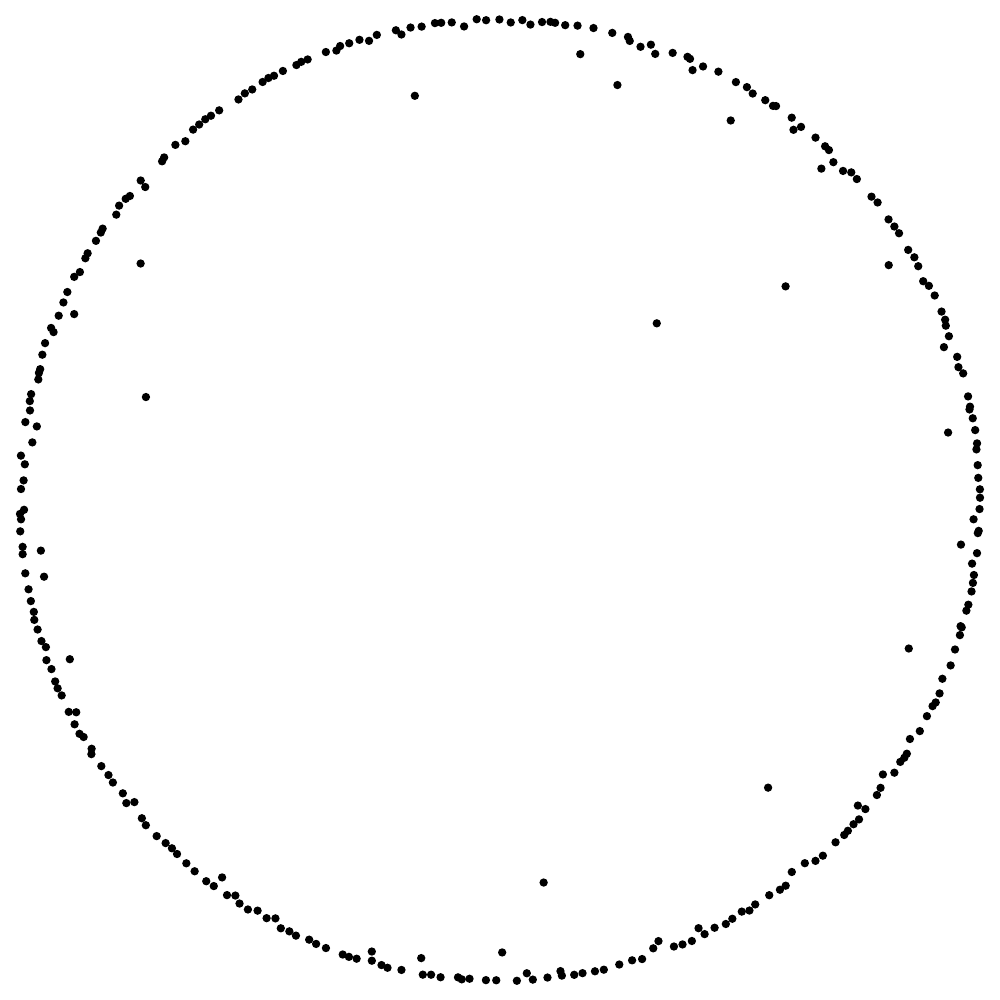}
\captionsetup{labelformat=empty}
\caption{(a)}
\addtocounter{figure}{-1}
\endminipage\hfill
\minipage{0.32\textwidth}
\includegraphics[width=\linewidth]{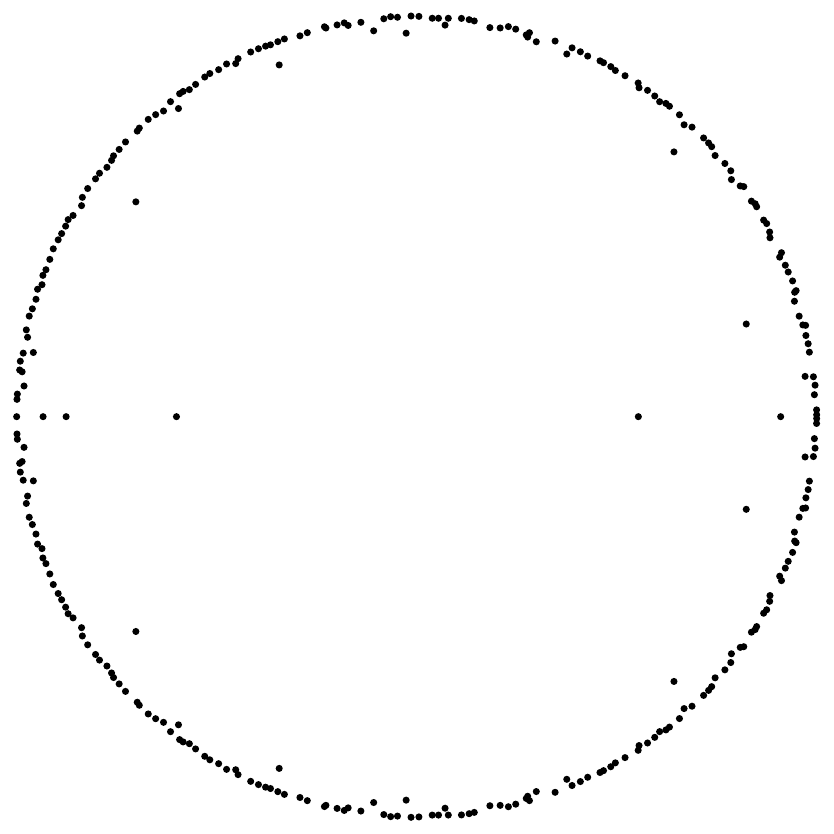}
\captionsetup{labelformat=empty}
\caption{(b)}
\addtocounter{figure}{-1}
\endminipage\hfill
\minipage{0.32\textwidth}
\includegraphics[width=\linewidth]{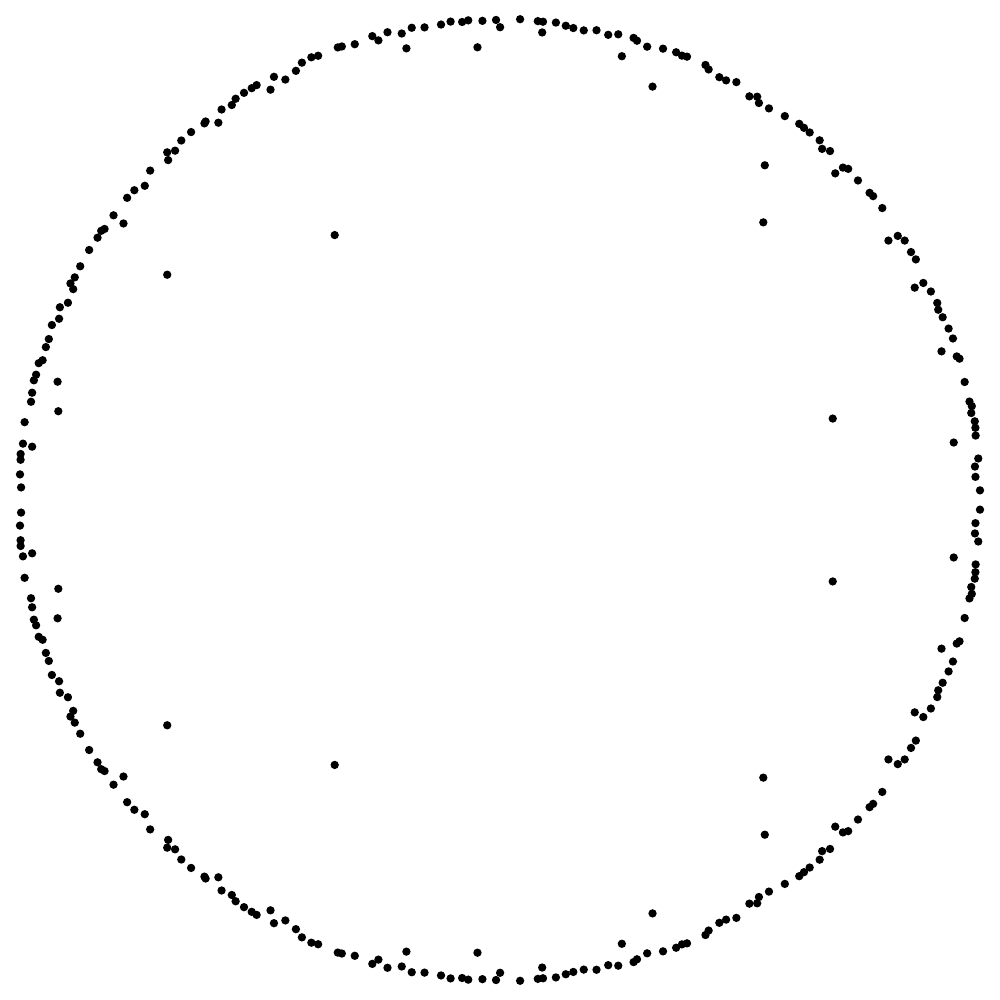}
\captionsetup{labelformat=empty}
\caption{(c)}
\addtocounter{figure}{-1}
\endminipage
\caption{A realization of random eigenvalues for truncations of: (a) $\bbU(n+1)$ with $n=301$; (b) $\bbO(n+1)$ with $n=301$; (c) $\bbUSp(n+1)$ with $n=151$.}
\end{figure}

\smallskip

Similarly, we develop matrix models for truncations of the orthogonal group (and, more generally, of the 
real orthogonal $\beta$-ensemble) 
and of the compact symplectic group. We establish connection to zeros of orthogonal polynomials and compute the distribution of the (independent) recurrence coefficients. Finally, we compute the distribution of eigenvalues of these ensembles (equivalently, distribution of zeros of these random orthogonal polynomials), with the exception of the eigenvalue distribution  of the truncated compact symplectic group, which remains an open problem as of now. For a reader who is wondering why the compact symplectic group seems special among all the other classical ensembles, we remark that deleting one quaternionic row and column corresponds to the deletion of \textit{two} complex rows and columns.
Truncation of more than 1 row and column can also in principle be attacked with our methods, which leads to a CMV model with \textit{matrix-valued} Verblunsky coefficients. However, the Jacobian computation is not technically manageable to us at this point. We leave this as an interesting open problem.

The joint eigenvalue distribution for truncated $\bbO(2k)$ was found earlier via ad hoc methods by Khoruzhenko--Sommers--{\.Z}yczkowski in~\cite{TruncOrthogonal}. Since the completion of this work, Forrester~\cite{For16} reported a formula for the joint eigenvalue density of  truncated $\bbUSp(n)$ ensemble.

Our methods can be applied to  random ensembles with the so-called non-ideal coupling (Fyodorov--Sommers~\cite{FyoSom00}, Fyodorov~\cite{Fyo01}, and Fyodorov--Khoruzhenko~\cite{FyoKho07}). We produce $\beta$-matrix models and compute the joint eigenvalue density.


After establishing the joint distribution, it is natural to ask what can be said about  the correlation functions and their large $n$ behavior. In our upcoming paper~\cite{KK_universality} we investigate the microscopic limit of the $1$-point correlation function of the circular $\beta$-ensemble ($0<\beta<\infty$). In fact, we show that the same limit is obtained for a wide class of subunitary random matrices.


The organization of the paper is as follows. In Section~\ref{sPreliminaries} we introduce the rest of our unitary random matrix ensembles (the compact groups) and define CMV matrices and the CMV-fication algorithm. In the rather lengthy Section~\ref{sSpectral} we compute the random spectral measure of each of the ensembles we study. Their eigenvalue distributions are well-known of course, but we also need distributions of the eigenweights of the spectral measures. In Section~\ref{sModels} we present the matrix models that are obtained by CMV-fying each of the (non-truncated) unitary ensembles. This is originally the idea of Killip--Nenciu~\cite{KN}, and we extend this to all the other unitary ensembles.

In Section~\ref{sModels2} we derive CMV matrix models for  each of the truncated ensembles. In Section~\ref{sEigenvalues} we compute the joint eigenvalue distribution of the truncations. In Section~\ref{sNonperfect} we discuss matrix models and compute the eigenvalue distribution for the non-ideally coupled ensembles.

In Section~\ref{sLoggas} we show that the eigenvalues of the truncated Dyson's (and, more generally, circular $\beta$-) ensembles admit a log-gas interpretation as charged particles located in a dielectric medium.

In Section~\ref{sSymmetricCMV} we introduce a symmetric variant of the CMV model, which allows us to find a canonical system of representatives of each conjugacy class that actually lie in the sets $COE$ and $CSE$; the traditional CMV matrix representation has no such symmetry properties. Finally, in Appendices we collect the information about quaternions, basics of the theory of orthogonal polynomials on the unit circle, and Jacobian determinants.


\section{Preliminaries}\label{sPreliminaries}

\subsection{Compact groups and Haar measure}\label{ssEnsembles}
As discussed in the Introduction, apart from Dyson's circular ensembles defined in Definition~\ref{DysonEnsembles}, we also study the classical compact groups. For the notation related to quaternions and quaternionic matrices, see Appendix~\ref{sQuaternions}.
\begin{definition} The unitary, orthogonal, special orthogonal, and compact symplectic groups are defined, respectively, by
\begin{align*}
 \bbU(n) &= \left\{ Q\in\bbC^{n\times n} : Q^\dagger Q = Q Q^\dagger = I_n    \right\}, \\
 \bbO(n)&= \left\{ Q\in\bbR^{n\times n} : Q^T Q = Q Q^T = I_n   \right\}, \\
 \bbSO(n)&= \left\{ Q\in\bbR^{n\times n} : Q^T Q = Q Q^T = I_n, \det Q =1   \right\}, \\
 \bbUSp(n)&= \left\{ Q\in\bbH_\bbR^{n\times n}  : \fC(Q)^\dagger \fC(Q) = \fC(Q) \fC(Q)^\dagger = I_{2n} \right\}.
\end{align*}
\end{definition}
Here $I_n$ is the $n\times n$ identity matrix; ${}^\dagger$ and ${}^T$ stands for the Hermitian conjugation and transposition, respectively.

By the general theory, each of these groups possesses the Haar measure, which is the unique (normalized by 1) measure that is invariant under the group action. We will use the same notation  $\bbU(n),\bbO(n),\bbSO(n),\bbUSp(n)$ for the group itself and for the random ensemble.

Next we describe the joint law of the columns of matrices from these groups.  This is well known, as is the method of proof we employ.  We give details only in the $\bbUSp$ case where the argument contains an additional subtlety.

\begin{proposition}\label{columns}
$(${\normalfont{a}}$)$ The first column $v_1$ of a matrix from  $\bbO(n)$  is distributed uniformly on the sphere $\{x\in\bbR^n : ||x|| =1 \}$.  For $2\le k\le n$, the $k$-th column $v_k$ is distributed uniformly on the subset of this unit sphere that is orthogonal to $v_1,\ldots,v_{k-1}$.

$(${\normalfont b}$)$ The first column $v_1$ of a matrix from $\bbU(n)$ is distributed uniformly on the sphere $\{ z\in\bbC^n:  ||z|| =1\}$. For $2\le k\le n$, the $k$-th column $v_k$ is distributed uniformly on the subset of this unit sphere that is orthogonal to $v_1,\ldots,v_{k-1}$.

$(${\normalfont  c}$)$ Let $Q\in\bbUSp(n)$. The first column $v_1$ of $\fC(Q)$ is distributed uniformly on the sphere $\{z\in\bbC^{2n}: ||z|| =1 \}$. The second column  $v_2$ is $Z\bar{v}_1$, where $Z$ is~\eqref{z}. For $2\le k\le n$, the $(2k-1)$-th column $v_{2k-1}$ is distributed uniformly on the subset of this unit sphere that is orthogonal to $v_1,\ldots,v_{2k-2}$, and $v_{2k}=Z\bar{v}_{2k-1}$.
\end{proposition}
\begin{proof}
As noted above, we give details only for the  case of $\bbUSp(n)$.

Form $V$ with columns $v_1,\ldots,v_{2n}$ described by the procedure in the Proposition, and let us show that $V$ is indeed $\fC(Q)$ for a $Q\in\bbUSp(n)$. By the construction, $v_1,\ldots,v_{2n}$ form an orthonormal basis of $\bbC^{2n}$ since $v \perp Z \bar{v}$ holds for any vector $v$. Thus $V$ is unitary. Moreover, $V$ consists of $n^2$ blocks each of which has the real quaternionic form, see~\eqref{E:RSU(2)}. Therefore $\fC^{-1}(V)$ belongs to the \textit{group} (not \textit{ensemble} yet) $\bbUSp(n)$. Let us show that for any non-random $W$ chosen from the group $\fC(\bbUSp(n))$, the distribution of $V$ and $W V$ are identical. This will prove that $\fC^{-1}(V)$ was generated according to the Haar measure by the uniqueness of the measure invariant with respect to the group action on $\bbUSp(n)$.

Denote the columns of $W V$ by $u_1,\ldots,u_{2n}$. First note that $u_1=W v_1$ is a vector uniformly distributed on the unit sphere $\{z\in\bbC^{2n}:  ||z|| =1 \}$ since $v_1$ is. The second column $u_2=W v_2 = W Z\bar{v}_1 = Z\overline{W v_1} = Z\bar{u}_1$ (we used~\eqref{realQuatDuality}), which agrees with our construction. Now inductively, consider $u_{2k-1}$, which is orthogonal to $u_1,\ldots,u_{2k-2}$. Consider any $2n\times 2n$ unitary $U$ that maps $\operatorname{span}\{u_1,\ldots,u_{2k-2}\}$ to itself. Then $W^{-1} U W$ maps $\operatorname{span}\{v_1,\ldots,v_{2k-2}\}$ to itself. By the construction this means that $v_{2k-1}$ and $W^{-1} U W v_{2k-1}$ are distributed identically. Thus $u_{2k-1}=W v_{2k-1}$ and $U u_{2k-1} = U W v_{2k-1}$ are distributed identically. Since $U$ was arbitrary, we conclude that $u_{2k-1}$ was distributed uniformly on the subspace of $\bbC^{2n}$ orthogonal to $u_1,\ldots,u_{2k-2}$.

Finally, $u_{2k}=Z\bar{u}_{2k-1}$ as before, and this completes the induction.
\end{proof}

\subsection{CMV matrices}\label{ssCMV}



Let us now give the definition of CMV matrices. They are named after Cantero--Moral--Velasquez~\cite{CMV} (also discovered earlier by Bunse-Gerstner--Elsner~\cite{CMV_BGE}). 
For a comprehensive discussion, see~\cite{S_CMV}.

\begin{definition}\label{CMV}
Given coefficients $\alpha_0,\ldots,\alpha_{n-1}$ with $\alpha_k \in \bbD$ for $0\le k \le n-2$ and $\alpha_{n-1}\in\overline{\bbD}$, let $\rho_k=\sqrt{1-|\alpha_k|^2}$, and define
\begin{equation}\label{CMV1}
\Xi_k=
\left[
\begin{array}{cc}
\bar{\alpha}_k & \rho_k \\
\rho_k & -\alpha_k
\end{array}
\right]
\end{equation}
for $0\le k\le n-1$, while $\Xi_{-1}=\left[ 1\right]$ and $\Xi_{n-1}=\left[ \bar{\alpha}_{n-1}\right]$. From these, form the $n\times n$ block-diagonal matrices
\begin{equation}\label{CMV2}
\calL = \operatorname{diag}\left( \Xi_0,\Xi_2,\Xi_4,\ldots\right),   \qquad
\calM = \operatorname{diag}\left( \Xi_{-1},\Xi_1,\Xi_3,\ldots\right).
\end{equation}
The matrix
\begin{equation*}\label{CMV3}
\calC(\alpha_0,\ldots,\alpha_{n-1}):=\calL \calM
\end{equation*}
is called the \textbf{CMV matrix} corresponding to $\alpha_0,\ldots,\alpha_{n-1}$.
\end{definition}
\begin{remark}
Note that $\calC(\alpha_0,\ldots,\alpha_{n-1})$ is unitary if and only if $\alpha_{n-1}\in\partial\bbD$. If $\alpha_{n-1}\in \bbD$, then it can be shown that $\calC(\alpha_0,\ldots,\alpha_{n-1})$ is a subunitary operator ($\calC^\dagger \calC \le 1$) with defect operator of order $1$ (see~\cite{bSzNagyFoias} as well as~\cite{AGT}). Usually the name ``CMV matrix'' is reserved for the cases when $\alpha_{n-1}\in \partial\bbD$, and if $\alpha_{n-1}\in \bbD$ then these matrices are referred to as the ``cut-off CMV matrices''. In this paper we will use the term ``CMV matrix'' for all $\alpha_{n-1}\in\overline{\bbD}$. 
\end{remark}

Let us now define what we will refer to here as the CMV-fication algorithm. Suppose we are given an $n\times n$ unitary matrix $U$ (with complex entries). If the vector $e_1:= [1,0,0,\ldots,0]^T$ is cyclic for $U$ then (see~\cite{CMV_BGE,CMV,S_CMV}) applying the Gram--Schmidt orthonormalization procedure to $e_1, U e_1, U^{-1} e_1,U^2 e_1, U^{-2} e_1,\ldots$ produces a basis in which $U$ has the CMV form $\calC$ with $\alpha_{n-1}\in\partial\bbD$.

\smallskip

Similar construction works if $U$ is instead a unitary  operator on $\ell^2(\bbZ_+)$ with the cyclic vector $e_1$. Then one ends up with an infinite CMV matrix $\calC(\alpha_0,\alpha_1,\ldots)$ with $\alpha_j\in\bbD$ for all $j\ge0$. This matrix is defined as $\calL \calM$, where $\calL$ and $\calM$ are semi-infinite block-diagonal matrices~\eqref{CMV2}, where $\Xi_k$ are as in \eqref{CMV1} for all $k\ge1$. 

\smallskip

Recall that by the spectral theorem, if $U$ is a unitary operator on $\ell^2(\bbC^n)$ or on $\ell^2(\bbZ_+)$ with a cyclic vector $e_1$, then its spectral measure is defined to be the probability measure on $\partial\bbD$ uniquely determined by
\begin{equation}\label{spectralMeasure}
\langle  e_1,U^m  e_1 \rangle = \int_0^{2\pi} e^{im\theta} d\mu(\theta),
\end{equation}
where $\langle v,u \rangle = v^\dagger u$. In fact, we have an isometry $\imath$ of Hilbert spaces $L^2(\mu)$ and $\ell^2(\bbC^n)$ or $\ell^2(\bbZ_+)$ defined by
\begin{equation}\label{isometry}
\imath: z^m \mapsto U^m e_1, \quad m\in\bbZ.
\end{equation}

\smallskip

Now note that during the Gram-Schmidt procedure above, $e_1$ was not changed, and therefore our CMV-fication algorithm can be described as
\begin{equation}\label{W}
U=W \calC W^\dagger
\end{equation}
for some unitary $W$ with $W e_1=W^\dagger e_1 = e_1$. This means that not only the eigenvalues of  $U$ and $\calC$ coincide, but also that the spectral measures~\eqref{spectralMeasure} of $U$ and $\calC$ with respect to $e_1$ are identical. In fact, two unitary CMV matrices with identical spectral measures necessarily coincide, see~\cite[Prop 3.3]{KN_CMV}.

Thus from the physical point of view, one may want to think that given a physical system modeled by a unitary matrix $U$ with a chosen vector $e_1$ (viewed as the initial state of the system) there is a unique choice of basis in which the matrix attains the ``minimal complexity'' CMV form.

Note however that in the case of systems with time-reversal invariance ($COE$ and $CSE$) the associated CMV form will not possess the same symmetry properties. In Section~\ref{sSymmetricCMV}, we introduce a variant of the CMV form that does possess the requisite symmetry.  Moreover, the matrix used to conjugate to this new canonical form belongs to the othogonal/symplectic group.

\smallskip

From the mathematical point of view, one of the advantages of reducing a unitary matrix or operator $U$ to the CMV form, apart from the latter being sparse and having the same the spectral measure, is that there is a direct connection to the theory of orthogonal polynomials on the unit circle. Indeed, characteristic polynomials of $\calC(\alpha_0,\ldots,\alpha_{k-1})$, $k=1,2,\ldots$, coincide with the orthogonal polynomials $\Phi_k$ associated to the spectral measure $\mu$ of $U$ with respect to $e_1$ (see Appendix~\ref{sOPUC}):
\begin{equation}\label{characteristicPoly}
\det(z I_k-\calC(\alpha_0,\ldots,\alpha_{k-1}))=\Phi_k(z).
\end{equation}
In particular, this means that these polynomials satisfy Szeg\H{o}'s recurrence~\eqref{OPUC}. We stress that the recurrence coefficients in~\eqref{OPUC} are precisely the coefficients from the CMV matrix. We will refer to them as the Verblunsky coefficients.

We collect all the needed facts from the theory of orthogonal polynomials in Appendix~\ref{sOPUC}. For more details the reader is referred to the monographs~\cite{OPUC1,OPUC2}.

\subsection{Block CMV matrices and matrix-valued spectral measures}
For the quaternionic ensembles $CSE(n)$ and $\bbUSp(n)$ we will have to deal with the $2\times 2$ matrix-valued spectral measures and with block CMV forms which we define next.

Recall that $\mu$ is an $l\times l$ matrix-valued probability measure on $\partial\bbD$ if it is a countably additive mapping from Borel subsets of $\partial\bbD$ to positive semi-definite $l\times l$ matrices, with the normalization condition $\mu(\partial\bbD)=I_l$. The $l\times l$ spectral measure of an $ln \times ln$ unitary matrix $U$ with respect to vectors $e_1,e_2,\ldots, e_l$ is the
$l \times l$ matrix-valued measure $d\mu$ with entries $d\mu_{jk}$ ($1\le j,k \le l$) determined by
\begin{equation}\label{matrixSpectralMeasure}
\langle  e_j,U^m e_k \rangle = \int_0^{2\pi} e^{im\theta} d\mu_{jk}(\theta).
\end{equation}

Block CMV matrices can be defined as before, by replacing the scalar Verblunsky coefficients $\alpha_j\in\bar\D$ by  $l \times l$ matrices satisfying $\alpha^\dagger \alpha \le I_l$ and $\rho_j = (I_l-\alpha_j^\dagger \alpha_j)^{1/2}$.

Under the natural cyclicity condition, any unitary $nl\times nl$ matrix $U$ is unitarily equivalent to a block CMV matrix $W^\dagger UW=\calC(\alpha_0,\ldots,\alpha_{n-1})$ with the last coefficient satisfying $\alpha_{n-1}^\dagger \alpha_{n-1} = I_{l\times l}$. Moreover, the $l\times l$ matrix-valued spectral measure~\eqref{matrixSpectralMeasure} with respect to $e_1,e_2,\ldots, e_l$ is preserved (which is to say, $W$ fixes $e_1,\ldots e_l$). 

Now, given a matrix $Q$ from $CSE(n)$ or $\bbUSp(n)$, by its spectral measure we will call the $2\times 2$ matrix-valued spectral measure of the $2n\times 2n$ complex matrix $\fC(Q)$ with respect to $e_1,e_2$. The matrix $\fC(Q)$ can be reduced (with probability $1$) to $\calC(\alpha_0,\ldots,\alpha_{n-1})$ with $2\times 2$ Verblunsky coefficients $\alpha_j$. The $2\times 2$ spectral measures of $\fC(Q)$ and $\calC$ with respect to $e_1,e_2$ coincide.

\section{Spectral Measures of Unitary Ensembles}\label{sSpectral}

In this section we will derive the distribution of the random spectral measures of each of the ensembles defined above.

Let us denote the Vandermonde determinant by
\begin{equation*}\label{Vandermonde}
\Delta(z_1,\ldots,z_n) = \prod_{1\le j < k \le n} (z_k-z_j).
\end{equation*}

\subsection{Dyson's circular ensembles}

\begin{proposition}\label{propSpecCircular}
The spectral measure of $COE(n)$ and of $CUE(n)$ is
\begin{equation}\label{finiteSpectralMeasure}
\sum_{j=1}^n \mu_j \delta_{e^{i\theta_j}},
\end{equation}
where
\begin{align}\label{domain1}
&0\le \theta_j < 2\pi \quad (1\le j \le n), \\
\label{domain2}
&\sum_{j=1}^{n} \mu_j =1, \quad \mu_j > 0 \quad (1\le j \le n),
\end{align}
and the joint distribution of $e^{i\theta_1},\ldots,e^{i\theta_n},\mu_1,\ldots,\mu_{n-1}$ is
\begin{equation}\label{circular}
\tfrac{1}{Z_{n,\beta}} \left| \Delta(e^{i\theta_1},\ldots,e^{i\theta_n}) \right|^\beta \frac{d\theta_1}{2\pi} \ldots \frac{d\theta_n}{2\pi} \times \tfrac{1}{Z'_{n,\beta}} \prod_{j=1}^n \mu_j^{\beta/2-1} d\mu_1\ldots d\mu_{n-1}
\end{equation}
with $\beta=1$ for $COE$ and $\beta=2$ for $CUE$. Here 
$Z_{n,\beta}, Z'_{n,\beta}$ are the normalization constants expressed in terms of the gamma function  as
$$
Z_{n,\beta} = \frac{\Gamma(\tfrac12 \beta n+1)}{\left[\Gamma(\tfrac12 \beta+1)\right]^n}, \quad Z'_{n,\beta} = \frac{\left[\Gamma(\tfrac12 \beta)\right]^n}{\Gamma(\tfrac12 \beta n)}.
$$

The $2\times 2$ spectral measure of $CSE(n)$ is
\begin{equation}\label{CSEmeasure}
\sum_{j=1}^n \begin{bmatrix} \mu_j & 0 \\ 0 & \mu_j \end{bmatrix} \delta_{e^{i\theta_j}},
\end{equation}
where $\mu_j$ and $\theta_j$ follow the law \eqref{circular} with $\beta=4$ on the domain \eqref{domain1}--\eqref{domain2}.
\end{proposition}
\begin{remark}
In fact, the probability distribution of the eigenvalues in~\eqref{circular} with any $0<\beta<\infty$ has a natural interpretation as the Gibbs distribution for the classical Coulomb gas on the circle at the inverse temperature $\beta$. This distribution for any $0<\beta<\infty$ was realized as the eigenvalue distribution of the family of random matrix models called circular $\beta$-ensembles in Killip--Nenciu~\cite{KN}. See Proposition~\ref{thmDysonCMV} below.
\end{remark}
\begin{proof}
The distribution of the eigenvalues is well-known (see Dyson's original paper~\cite{Dyson}).

Distribution of the eigenweights for $\beta=2$ was computed in~\cite[Proposition~3.1]{KN}.

Let us compute the distributions of the eigenweights for $\beta = 1$. Suppose $U\in COE(n)$, and let its eigenvalues $e^{i\theta_1},\ldots,e^{i\theta_n}$ be given.
It can diagonalized by a \textit{real} orthogonal matrix $U=ODO^T$  (see~\cite[Thm~4]{Dyson}; or, more generally,~\cite[Thm~4.4.7]{bHornJohnson}). 
The matrix $O$ of eigenvectors is unique up to a multiplication of it by a diagonal matrix $S$ with $\pm1$ on its diagonal. If we choose each of these signs by random independently with equal probability, then the map $U\mapsto OS$ becomes well defined, and $OS$ becomes a Haar distributed orthogonal matrix. Indeed, the invariance property of $COE(n)$ (\cite[Thm~10.1.1]{bMehta}) implies that the distribution of $OS$ and $WOS$ are identical for any orthogonal $W$, and now we can apply the uniqueness of the Haar measure on $\bbO(n)$.
Now Proposition~\ref{columns}(a) says that the first row $x=(x_1,\ldots,x_n)$ of $OS$ is uniformly distributed on the unit sphere $S^{n-1}=\{{x} \in\bbR^{n}: ||x||=1 \}$. But by~\eqref{spectralMeasure} and~\eqref{finiteSpectralMeasure}, $\mu_j = x_j^2$. Now Lemma~\ref{lemSphere}(a) finishes the proof.

A similar approach works for $\beta=4$. Let $U\in CSE(n)$. Suppose the $2n$ eigenvalues of $\fC(U)$ are $e^{i\theta_1},e^{i\theta_1},\ldots,e^{i\theta_n},e^{i\theta_n}$ (each repeated twice). By~\cite[Thm~3]{Dyson} $U$ can be reduced to the diagonal form $U=B D B^R$ by a symplectic matrix $B\in\bbUSp(n)$, where $\fC(D)$ is the $2n\times 2n$ complex diagonal matrix with the diagonal $(e^{i\theta_1},e^{i\theta_1} ,\ldots,e^{i\theta_n},e^{i\theta_n})$. $B$ is defined uniquely up to multiplication by a diagonal matrix $S$ with unimodular real quaternions on its diagonal. If we choose those by random independently and uniformly on $\bbUSp(2)$, then by the invariance property of $CSE(n)$ (see~\cite[Thm~10.2.1]{bMehta}), $BS$ becomes Haar distributed random matrix from $\bbUSp(n)$.

Denote the first quaternionic column of $(BS)^R$ by $[q_1,\ldots,q_n]^T$ with
$$
\fC(q_j)=
\begin{bmatrix} \alpha_j & -\bar\beta_j \\ \beta_j & \bar{\alpha}_j\end{bmatrix}.
$$
with $\alpha_j,\beta_j\in\bbC$, $1\le j \le n$.
Now $U^m=(B S) D^m (BS)^R$ and a simple calculation implies
\begin{align*}
\langle  e_1,\fC(U)^m  e_1 \rangle &= \langle  e_2,\fC(U)^m  e_2 \rangle = \sum_{j=1}^n (|\alpha_j|^2+|\beta_j|^2) e^{im\theta_j}, \\
\langle  e_1,\fC(U)^m  e_2 \rangle &= \langle  e_2,\fC(U)^m  e_1 \rangle = 0.
\end{align*}
Therefore the $2\times 2$ spectral measure~\eqref{matrixSpectralMeasure} of $U$ is indeed of the form~\eqref{CSEmeasure} with $\mu_j = |\alpha_j|^2+|\beta_j|^2$. By Proposition~\ref{columns}(c), the first complex column of $\fC((BS)^R)$, that is $[\alpha_1,\beta_1,\ldots,\alpha_n,\beta_n]^T$, is uniformly distributed on the unit sphere of $\bbC^{2n}$. Then by Lemma~\ref{lemSphere}(e), $\mu_j$'s are distributed according to $(2n-1)! \prod_{j=1}^n \mu_j d\mu_1\ldots d\mu_{n-1}$ on~\eqref{domain2} as required.
\end{proof}

\subsection{Orthogonal group}
\begin{proposition}\label{thmOSpectr}
\begin{itemize}
\item[(a)]
The spectral measure of $\bbSO(2n)$ is
\begin{equation}\label{formSpectralMeasureO1}
\sum_{j=1}^n \tfrac12 \mu_j \left( \delta_{e^{i\theta_j}}+\delta_{e^{-i\theta_j}} \right),
\end{equation}
where
\begin{align}\label{domainO1}
&0 < \theta_j < \pi \quad (1\le j \le n), \\
\label{domainO2}
&\sum_{j=1}^{n} \mu_j =1, \quad \mu_j > 0 \quad (1\le j \le n),
\end{align}
 and the joint distribution of $e^{i\theta_1},\ldots,e^{i\theta_n},\mu_1,\ldots,\mu_{n-1}$ is
\begin{equation*}\label{SOeven}
\tfrac{1}{2^{n-1} n! \pi^n} \left| \Delta(2\cos\theta_1,\ldots,2\cos\theta_n) \right|^2 \, d\theta_1 \ldots d\theta_n 
\times (n-1)! d\mu_1\ldots d\mu_{n-1}.
\end{equation*}
\item[(b)] The spectral measure of $\bbO(2n)\setminus\bbSO(2n)$ is
\begin{equation}\label{formSpectralMeasureO2}
\sum_{j=1}^{n-1} \tfrac12\mu_j \left( \delta_{e^{i\theta_j}}+\delta_{e^{-i\theta_j}} \right) +\mu_{n} \delta_{1}+\mu_{n+1} \delta_{-1},
\end{equation}
where
\begin{align}\label{domainO3}
&0 < \theta_j < \pi \quad (1\le j \le n-1), \\
\label{domainO4}
&\sum_{j=1}^{n+1} \mu_j =1, \quad \mu_j > 0 \quad (1\le j \le n+1),
\end{align}
 and the joint distribution of $e^{i\theta_1},\ldots,e^{i\theta_{n-1}},\mu_1,\ldots,\mu_{n}$ is
\begin{multline*}\label{SOeven2}
\tfrac{2^{n-1}}{(n-1)!\pi^{n-1}}  \left| \Delta(2\cos\theta_1,\ldots,2\cos\theta_{n-1}) \right|^2 \prod_{j=1}^{n-1} \sin^2\theta_j \, d\theta_1 \ldots d\theta_{n-1} 
 \\
  \times
\tfrac{(n-1)!}{\pi}
  \tfrac{1}{\sqrt{\mu_n \mu_{n+1}}}  d\mu_1\ldots d\mu_n.
\end{multline*}
\item[(c)] The spectral measure of $\bbSO(2n+1)$ is
\begin{equation}\label{formSpectralMeasureO3}
\sum_{j=1}^n \tfrac12\mu_j \left( \delta_{e^{i\theta_j}}+\delta_{e^{-i\theta_j}} \right) +\mu_{n+1} \delta_1,
\end{equation}
with \eqref{domainO1}, \eqref{domainO4},
 and the joint distribution of $e^{i\theta_1},\ldots,e^{i\theta_n},\mu_1,\ldots,\mu_{n}$ is
\begin{equation*}\label{SOodd}
\tfrac{2^n}{n!\pi^n}  \left| \Delta(2\cos\theta_1,\ldots,2\cos\theta_n) \right|^2 \prod_{j=1}^n \sin^2\bigl(\tfrac{\theta_j}{2}\bigr) \, d\theta_1 \ldots d\theta_n
  \times
\tfrac{\Gamma(n+\tfrac{1}{2})}{\Gamma(\tfrac12)}
 \tfrac{1}{\sqrt{\mu_{n+1}}}  d\mu_1\ldots d\mu_n.
\end{equation*}
\item[(d)] The spectral measure of $\bbO(2n+1)\setminus\bbSO(2n+1)$ is
\begin{equation}\label{formSpectralMeasureO4}
\sum_{j=1}^n \tfrac12\mu_j \left( \delta_{e^{i\theta_j}}+\delta_{e^{-i\theta_j}} \right) +\mu_{n+1} \delta_{-1},
\end{equation}
with \eqref{domainO1}, \eqref{domainO4}, and the joint distribution of $e^{i\theta_1},\ldots,e^{i\theta_n},\mu_1,\ldots,\mu_{n}$ is
\begin{equation*}\label{SOodd2}
\tfrac{2^n}{n!\pi^n}  \left| \Delta(2\cos\theta_1,\ldots,2\cos\theta_n) \right|^2 \prod_{j=1}^n \cos^2\bigl(\tfrac{\theta_j}{2}\bigr) \, d\theta_1 \ldots d\theta_n
  \times
\tfrac{\Gamma(n+\tfrac{1}{2})}{\Gamma(\tfrac12)}
 \tfrac{1}{\sqrt{\mu_{n+1}}}  d\mu_1\ldots d\mu_n.
\end{equation*}
\end{itemize}
\end{proposition}
\begin{proof}
The distributions of the eigenvalues is well-known and can be found in, e.g.,~\cite{bWeyl} and~\cite{bPastur}, so we just need to justify the distribution of the eigenweights.

(a) is proven in~\cite[Prop 3.4]{KN}.

(b), (c), and (d) can be proved in the similar fashion. Indeed, one can reduce a random matrix to the block diagonal
\begin{align*}
&\operatorname{diag}\left( \begin{bmatrix} \cos\theta_1 & \sin\theta_1 \\ -\sin\theta_1 & \cos\theta_1 \end{bmatrix}, \ldots , \begin{bmatrix} \cos\theta_{n-1} & \sin\theta_{n-1} \\ -\sin\theta_{n-1} & \cos\theta_{n-1} \end{bmatrix} ,\begin{bmatrix} 1  \end{bmatrix},\begin{bmatrix} -1 \end{bmatrix} \right), \\
&\operatorname{diag}\left( \begin{bmatrix} \cos\theta_1 & \sin\theta_1 \\ -\sin\theta_1 & \cos\theta_1 \end{bmatrix}, \ldots , \begin{bmatrix} \cos\theta_n & \sin\theta_n \\ -\sin\theta_n & \cos\theta_n \end{bmatrix},\begin{bmatrix} 1 \end{bmatrix} \right), \\
&\operatorname{diag}\left( \begin{bmatrix} \cos\theta_1 & \sin\theta_1 \\ -\sin\theta_1 & \cos\theta_1 \end{bmatrix}, \ldots , \begin{bmatrix} \cos\theta_n & \sin\theta_n \\ -\sin\theta_n & \cos\theta_n \end{bmatrix} ,\begin{bmatrix} -1 \end{bmatrix} \right),
\end{align*}
respectively.

This shows that the vector of eigenweights $(\mu_1,\ldots,\mu_{n+1})$ for the cases (b), (c), and (d) is distributed as in Lemma~\ref{lemSphere} (c), (d), (d), respectively, which finishes the proof.
%
%
\end{proof}

\subsection{Compact symplectic unitary group}
\begin{proposition}
The $2\times 2$ spectral measure of $\bbUSp(n)$ is
\begin{equation}\label{USpmeasure}
\sum_{j=1}^n W_j \delta_{e^{i\theta_j}} +
\sum_{j=1}^n W_j^R \delta_{e^{-i\theta_j}},
\end{equation}
where
\begin{equation*}\label{matrixWeights}
W_j=\begin{bmatrix} \mu_j & \sqrt{\mu_j \nu_j} \, e^{i\psi_j} \\ \sqrt{\mu_j \nu_j} \, e^{-i\psi_j} & \nu_j \end{bmatrix},
\quad
W_j^R=\begin{bmatrix} \nu_j & - \sqrt{\mu_j \nu_j} \, e^{i\psi_j} \\ -\sqrt{ \mu_j \nu_j} \, e^{-i\psi_j} & \mu_j \end{bmatrix}.
\end{equation*}
Here
\begin{equation*}
\begin{aligned}
&0 < \theta_j < \pi \quad (1\le j \le n), \\
&0\le \psi_j < 2\pi \quad (1\le j \le n), \\
&\sum_{j=1}^{n} \mu_j + \sum_{j=1}^{n} \nu_j =1, \quad \mu_j > 0, \nu_j > 0 \quad (1\le j\le n),
\end{aligned}
\end{equation*}
and the joint distribution of $e^{i\theta_1},\ldots,e^{i\theta_n},e^{i\psi_1},\ldots,e^{i\psi_n},\mu_1,\ldots,\mu_{n},\nu_1,\ldots,\nu_{n-1}$ is
\begin{multline*}
\tfrac{2^{n}}{n!\pi^n}  \left| \Delta(2\cos\theta_1,\ldots,2\cos\theta_{n}) \right|^2 \prod_{j=1}^{n} \sin^2\theta_j \, d\theta_1 \ldots d\theta_n
 \\
 \times
  \tfrac{d\psi_1}{2\pi} \ldots \tfrac{d\psi_{n}}{2\pi}
 \times
(2n-1)! \, d\mu_1\ldots d\mu_n d\nu_1\ldots d\nu_{n-1}.
\end{multline*}
\end{proposition}
\begin{proof}
The distribution of the eigenvalues is a classical result. As for the matrix eigenweights, first write a matrix $U\in \bbUSp(n)$ as $U=BDB^R$, where $\fC(D)=\operatorname{diag}(e^{i \theta_1},e^{-i\theta_1},\ldots,e^{i\theta_n},e^{-i\theta_n})$. Repeating the arguments from the proof of Proposition~\ref{propSpecCircular}, we may think of $B$ as being Haar distributed in $\bbUSp(n)$. Let the first quaternionic column of $B^R$ be $[q_1,\ldots,q_n]^T$, $q_j\in\bbHR$, with
$$
\fC(q_j)=
\begin{bmatrix} \alpha_j & -\bar\beta_j \\ \beta_j & \bar{\alpha}_j\end{bmatrix}.
$$
Now $U^m=B D^m B^R$ and a simple calculation imply
\begin{align*}
\langle  e_1,\fC(U)^m  e_1 \rangle &= \sum_{j=1}^n |\alpha_j|^2 e^{im\theta_j} + \sum_{j=1}^n |\beta_j|^2 e^{-im\theta_j}, \\
\langle  e_2,\fC(U)^m  e_2 \rangle &= \sum_{j=1}^n |\beta_j|^2 e^{im\theta_j} + \sum_{j=1}^n |\alpha_j|^2 e^{-im\theta_j}, \\
\langle  e_1,\fC(U)^m  e_2 \rangle &= -\sum_{j=1}^n \bar\alpha_j \bar\beta_j e^{im\theta_j} + \sum_{j=1}^n \bar\alpha_j \bar\beta_j e^{-im\theta_j}, \\
\langle  e_2,\fC(U)^m  e_1 \rangle &= -\sum_{j=1}^n \alpha_j\beta_j e^{im\theta_j} + \sum_{j=1}^n \alpha_j\beta_j e^{-im\theta_j}.
\end{align*}
Therefore the matrix-valued measure~\eqref{matrixSpectralMeasure} has the form~\eqref{USpmeasure} with
$$
W_j=\begin{bmatrix} |\alpha_j|^2 & -\bar\alpha_j\bar\beta_j \\ -\alpha_j\beta_j & |\beta_j|^2 \end{bmatrix}.
$$
Using Proposition~\ref{columns}(c) 
and Lemma~\ref{lemSphere}(b),  we obtain the promised joint distribution.
\end{proof}


\section{Matrix Models of Unitary Random Matrix Ensembles}\label{sModels}

It turns out that if one applies the CMV-fication procedure described in Section~\ref{ssCMV}, then the Verblunsky coefficients become statistically independent with the distributions that are explicitly computable. For $CUE(n)$ and $SO(2n)$ this was done by Killip--Nenciu in~\cite{KN}. We restate these results here for completeness and develop analogues for the other ensembles.  We stress that for the quaternionic ensembles $CSE$ and $\bbUSp$, one needs to consider the $2\times2$ block CMV-fication if one hopes the same phenomenon to happen.

\begin{definition}\label{theta}
\begin{itemize}
\item[(a)] We say that a real-valued random variable, $X$, is beta-distributed with parameters $s,t>0$, which we denote by $X\sim B(s,t)$ if
\begin{equation*}
\bbE\{f(X)\} = \tfrac{2^{1-s-t}\Gamma(s+t)}{\Gamma(s)\Gamma(t)} \int_{-1}^1 f(x) (1-x)^{s-1}(1+x)^{t-1} d x.
\end{equation*}
We write $X\sim B(0,0)$ when $X$ takes values $1$ and $-1$ with probabilities $1/2$.  (This is the $t=s\to0$ limit of the above.)

\item[(b)] We say that a complex random variable, $X$, with values in the unit disk $\bbD$ is $\Theta(\nu)$-distributed (for $\nu>1$) if
\begin{equation*}
\bbE\{f(X)\} = \tfrac{\nu-1}{2\pi} \int\int_D f(z) (1-|z|^2)^{(\nu-3)/2} d^2 z,
\end{equation*}
where $d^2 z$ is the Lebesgue measure on $\bbD$.  As a limit case, $X\sim\Theta(1)$ means that $X$ is uniformly distributed on the unit circle.

\item[(c)] We say that a $2\times 2$ matrix-valued random variable, $X$, is $\Upsilon(\nu)$-distributed (for $\nu>3$) if
\begin{equation}\label{xMatrix}
X=
\begin{bmatrix} a_0+a_1 i & -a_2+a_3 i \\ a_2+a_3 i  & a_0-a_1 i\end{bmatrix},
\end{equation}
where, writing $\vec a = (a_1,a_2,a_3,a_4)$,
\begin{align*}
\bbE\{f(\vec a)\}  = \tfrac{(\nu-1)(\nu-3)}{4\pi^2} \int_{||\vec x||\le1} f(\vec x) (1-|\vec x|^2)^{(\nu-5)/2} d x_1 dx_2 dx_3 dx_4.
\end{align*}
By extension, we write $X\sim\Upsilon(3)$ when $X$ is Haar distributed in $SU(2)$, that is,~\eqref{xMatrix} with $(a_1,a_2,a_3,a_4)$ uniformly distributed over $\{x\in\bbR^4 :x_1^2+x_2^2+x_3^2+x_4^2=1\}$.
\end{itemize}
\end{definition}
\begin{remarks}
1. Incorporating our stated extensions, $B(\nu/2,\nu/2)$ is the distribution of the first component of the random vector uniformly distributed over the $\nu$-sphere (if $\nu\ge0$). $\Theta(\nu)$ is the distribution of $x_1+ix_2$, where $x_1,x_2$ are the first two components of the random vector uniformly distributed over the $\nu$-sphere (if $\nu\ge1$). Finally, the  $\Upsilon(\nu)$-distributed $a_1,a_2,a_3,a_4$ appear as the  first four components of the random vector uniformly distributed over the $\nu$-sphere  (if $\nu\ge3$). This is shown in~\cite[Cor~A.2]{KN}.
\end{remarks}

\begin{proposition}\label{thmDysonCMV}
The CMV-fication of ensembles $COE(n)$ or $CUE(n)$ produces the CMV matrix  $\calC(\alpha_0,\ldots,\alpha_{n-1})$  with independent Verblunsky coefficients
\begin{equation}\label{dysonCMV}
\alpha_k \sim \Theta(\beta(n-1-k)+1), \quad 0\le k \le n-1,
\end{equation}
where $\beta=1$ for $COE(n)$ and $\beta=2$ for $CUE(n)$.

More generally, for any $0<\beta<\infty$, the spectral measure of $\calC(\alpha_0,\ldots,\alpha_{n-1})$ with independent Verblunsky coefficients \eqref{dysonCMV} has the law  \eqref{circular}.
\end{proposition}

\begin{remarks}
1. The ensembles of CMV matrices with~\eqref{dysonCMV} for any $0<\beta<\infty$ were introduced by Killip--Nenciu in~\cite{KN}, and are called the circular $\beta$-ensembles. One way to look at it is that Dyson's ensembles embed themselves into a family of random matrix ensembles that depend on $0<\beta<\infty$ continuously. This looks natural if one thinks of the eigenvalue point process as an interacting particle system.  Similarly, we will see a $\beta$-extrapolation for $\bbSO(n)$ and $\bbO(n)\setminus\bbSO(n)$ ensembles below as well. It would be interesting to find a natural $\beta$-extrapolation for $\bbUSp(n)$.

2. The CMV-fication procedure requires cyclicity of $e_1$ or, equivalently, that $|\alpha_k|<1$ for $k<n-2$.  This holds with probability 1.
\end{remarks}
\begin{proof}
The result for $CUE(n)$ was established in~\cite[Prop 3.3]{KN}.

The result for $COE(n)$ is the combination of our Proposition~\ref{propSpecCircular} above,~\cite[Prop 4.2]{KN}, and the fact that two CMV matrices with the same spectral measures must coincide.

The last statement of Proposition is in~\cite[Prop 4.2]{KN}.
\end{proof}

\begin{proposition}\label{thmCSECMV}
The $2\times 2$ block CMV-fication of $CSE(n)$ produces the block CMV matrix $\calC(\alpha_0 I_2,\ldots,\alpha_{n-1} I_2)$ with independent matrix-valued Verblunsky coefficients $\alpha_0 I_2,\ldots,\alpha_{n-1} I_2$, where  $\alpha_j$ are distributed according to~\eqref{dysonCMV} with $\beta=4$.

In other words, the $2\times 2$ spectral measures of $CSE(n)$ and  of $\calC(\alpha_0 I_2,\ldots,\alpha_{n-1} I_2)$ with $\alpha_j$ given by \eqref{dysonCMV} with $\beta=4$ coincide.
\end{proposition}
\begin{remark}
Scalar CMV-fication of $CSE(n)$ is not natural.  It would produce CMV matrix with $2n$ Verblunsky coefficients with non-independent distributions.
\end{remark}
\begin{proof}
By Proposition~\ref{propSpecCircular} the random $2\times2$ spectral  measure of a $CSE(n)$ matrix is equal to $I_2$ times the scalar measure~\eqref{finiteSpectralMeasure} with \eqref{circular} and $\beta=4$. This implies that the $2\times2$ Verblunsky coefficients of $CSE(n)$ are equal to $I_2$ times the scalar Verblunsky coefficients of the scalar measure, which are~\eqref{dysonCMV} by the previous proposition.
\end{proof}

\begin{proposition}\label{thmOCMV}
The CMV-fication of the $\bbO(N)$ ensemble produces the CMV matrix  $\calC(\alpha_0,\ldots,\alpha_{N-1})$  with independent Verblunsky coefficients distributed as
\begin{equation}\label{bbOCMV}
\alpha_k \sim B\left( \tfrac{N-1-k}{2}, \tfrac{N-1-k}{2}\right), \quad 0\le k \le N-1.
\end{equation}
\end{proposition}
\begin{remark}
CMV-fication of $\bbSO(n)$ and $\bbO(n)\setminus\bbSO(n)$ gives the same~\eqref{bbOCMV} except for the last coefficient $\alpha_{n-1}$ which has to be taken always $-1$ for $\bbSO(2k)$ or $\bbO(2k+1)\setminus\bbSO(2k+1)$ or $1$ for $\bbSO(2k+1)$ or $\bbO(2k)\setminus\bbSO(2k)$.
\end{remark}
\begin{proof}
This was proved for $\bbSO(2n)$ in~\cite[Prop 3.5]{KN}. The other three cases can be proved in the exact same way. Finally note that the total measure of $\bbSO(N)$ and $\bbO(N)\setminus \bbSO(N)$ with respect to the Haar measure on $\bbO(N)$ are equal. This gives the $B(0,0)$ distribution for $\alpha_{N-1}$.
\end{proof}

The orthogonal groups treated in Proposition~\ref{thmOCMV} can be realized as special cases of the following more general result:

\begin{proposition}[Real orthogonal $\beta$-ensemble]\label{thmRealOrtho}
For any $0<\beta<\infty$, $a>-1$, and $b>-1$, we have:
\begin{itemize}
\item[(a)] ($N=2n, \det=1$) The spectral measure of $\calC(\alpha_0,\ldots,\alpha_{2n-2},-1)$ with independent Verblunsky coefficients
\begin{equation*}\label{generalSOCMV}
\alpha_k \sim
\begin{cases}
B\left( \frac{2n-2-k}{4} \beta + a+1, \frac{2n-2-k}{4}\beta + b+1  \right), & 0\le k \le 2n-2, k \mbox{ even}, \\
B\left( \frac{2n-3-k}{4} \beta + a+b+2, \frac{2n-1-k}{4}  \beta \right), & 0\le k \le 2n-2, k \mbox{ odd}
\end{cases}
\end{equation*}
is~\eqref{formSpectralMeasureO1} on~\eqref{domainO1}, \eqref{domainO2} with the joint distribution
\begin{multline*}
\tfrac{1}{C_n(a,b,\beta)} \left| \Delta(2\cos\theta_1,\ldots,2\cos\theta_n) \right|^\beta \prod_{j=1}^n |1-\cos\theta_j|^{a+\tfrac12} |1+\cos\theta_j|^{b+\tfrac12} \, d\theta_1 \ldots d\theta_n
\\ \times \tfrac{1}{K_n(\beta)} \prod_{j=1}^n \mu_j^{\tfrac\beta2-1} d\mu_1\ldots d\mu_{n-1}.
\end{multline*}

\item[(b)] ($N=2n, \det=-1$) The spectral measure of $\calC(\alpha_0,\ldots,\alpha_{2n-2},1)$ with independent Verblunsky coefficients
\begin{equation}\label{generalSOCMV2}
\alpha_k \sim
B\left( \tfrac{2n-1-k}{4} \beta , \tfrac{2n-1-k}{4}\beta  \right), \quad 0\le k \le 2n-2
\end{equation}
is~\eqref{formSpectralMeasureO2} on~\eqref{domainO3}, \eqref{domainO4} with the joint distribution
\begin{multline*}
\tfrac{1}{D_n(\beta)} \left| \Delta(2\cos\theta_1,\ldots,2\cos\theta_{n-1}) \right|^\beta \prod_{j=1}^{n-1} |1-\cos\theta_j|^{\tfrac{3\beta}{4}-\tfrac12} |1+\cos\theta_j|^{\tfrac{3\beta}{4}-\tfrac12} \, d\theta_1 \ldots d\theta_n
\\ \times \tfrac{1}{L_n(\beta)} (\mu_n\mu_{n+1})^{\tfrac\beta4-1} \prod_{j=1}^{n-1} \mu_j^{\tfrac\beta2-1}  d\mu_1\ldots d\mu_{n}.
\end{multline*}

\item[(c)] ($N=2n+1, \det=1$) The spectral measure of $\calC(\alpha_0,\ldots,\alpha_{2n-1},1)$ with independent Verblunsky coefficients
\begin{equation*}\label{generalSOCMV3}
\alpha_k \sim
B\left( \tfrac{2n-k}{4} \beta , \tfrac{2n-k}{4}\beta  \right), \quad 0\le k \le 2n-1
\end{equation*}
is~\eqref{formSpectralMeasureO3} on~\eqref{domainO1}, \eqref{domainO4} with the joint distribution
\begin{multline*}
\tfrac{1}{E_n(\beta)} \left| \Delta(2\cos\theta_1,\ldots,2\cos\theta_n) \right|^\beta \prod_{j=1}^n |1-\cos\theta_j|^{\tfrac{3\beta}{4}-\tfrac12} |1+\cos\theta_j|^{\tfrac\beta4-\tfrac12} \, d\theta_1 \ldots d\theta_n
\\ \times \tfrac{1}{M_n(\beta)} \mu_{n+1}^{\tfrac\beta4-1} \prod_{j=1}^{n}  \mu_j^{\tfrac\beta2-1}  d\mu_1\ldots d\mu_{n}.
\end{multline*}

\item[(d)] ($N=2n+1, \det=-1$) The spectral measure of $\calC(\alpha_0,\ldots,\alpha_{2n-1},-1)$ with independent Verblunsky coefficients
\begin{equation*}\label{generalSOCMV4}
\alpha_k \sim
B\left( \tfrac{2n-k}{4} \beta , \tfrac{2n-k}{4}\beta  \right), \quad 0\le k \le 2n-1
\end{equation*}
is~\eqref{formSpectralMeasureO4} on~\eqref{domainO1}, \eqref{domainO4} with the joint distribution
\begin{multline*}
\tfrac{1}{E_n(\beta)} \left| \Delta(2\cos\theta_1,\ldots,2\cos\theta_n) \right|^\beta \prod_{j=1}^n |1-\cos\theta_j|^{\tfrac{\beta}{4}-\tfrac12} |1+\cos\theta_j|^{\tfrac{3\beta}{4}-\tfrac12} \, d\theta_1 \ldots d\theta_n
 \\ \times \tfrac{1}{M_n(\beta)} \mu_{n+1}^{\tfrac\beta4-1} \prod_{j=1}^{n}  \mu_j^{\tfrac\beta2-1}  d\mu_1\ldots d\mu_{n}.
\end{multline*}
\end{itemize}

The normalization constants are:
\begin{align*}
 C_n(a,b,\beta) &=  n! 2^{n(a+b+1)+\beta n (n-1)} \prod_{j=0}^{n-1} \frac{\Gamma(a+1+\tfrac\beta2 j)\Gamma(b+1+\tfrac\beta2 j)\Gamma(\tfrac\beta2 (j+1))}{\Gamma(a+b+2+\tfrac\beta2 (n-1+j)) \Gamma(\tfrac\beta2)}; \\
 K_n(\beta) &= \frac{\left[\Gamma(\frac12 \beta)\right]^n}{\Gamma(\frac12 \beta n)};\\
 D_n(\beta) &= \frac{1}{L_n(\beta) } \frac{n! \pi^{n-1/2}}{2^{(2n-1)\frac\beta2-n}} \prod_{j=1}^{2n-1} \frac{\Gamma(\frac\beta4 j)}{\Gamma(\frac12 + \frac\beta4 j)}; \\
 L_n(\beta) &= \frac{\left[\Gamma(\frac12 \beta)\right]^{n-1} \left[\Gamma(\frac14 \beta)\right]^2 }{\Gamma(\frac12 \beta n)};\\
 E_n(\beta) &= \frac{1}{M_n(\beta) } \frac{n! \pi^n}{2^{(\frac\beta2-1)n}} \prod_{j=1}^{2n} \frac{\Gamma(\frac\beta4 j)}{\Gamma(\frac12 + \frac\beta4 j)}; \\
 M_n(\beta) &= \frac{\left[\Gamma(\frac12 \beta)\right]^{n} \Gamma(\frac14 \beta) }{\Gamma(\frac12 \beta (n+\tfrac12))}.
\end{align*}
\end{proposition}
\begin{remarks}
1. If $\beta=2, a=b=-1/2$ in (a), and $\beta=2$ in (b), (c), (d), then one obtains matrix models for $\bbSO(2n)$, $\bbO(2n)\setminus \bbSO(2n)$, $\bbSO(2n+1)$, $\bbO(2n+1)\setminus \bbSO(2n+1)$, respectively.

2. The ensemble of matrices in (a) was introduced in~\cite{KN}. It is now called the real orthogonal $\beta$-ensemble (see Forrester~\cite[Sect 2.9]{bForrester}).

3. This admits a log-gas interpretation. Indeed, the eigenvalues of (a) can be viewed as $n$ unit charges $e^{i\theta_1},\ldots, e^{i\theta_n}$ lying on the half-circle $0<\theta<\pi$ that also interact with their image charges at $e^{-i\theta_1},\ldots, e^{-i\theta_n}$ and with the fixed charges $-\tfrac12+\tfrac{2a+1}{\beta}$  at $1$ and $-\tfrac12+\tfrac{2b+1}{\beta}$  at $1$. As usual, $0<\beta<\infty$ plays the role of the inverse temperature. This is worked out in~\cite[Prop 2.9.6]{bForrester}. Similarly, one can check that (b) corresponds to the charges $1-\tfrac1\beta$ at $1$ and $-1$; (c) to the charge $1-\tfrac1\beta$ at $1$ and $-\tfrac1\beta$ at $-1$; (d) to the charge $-\tfrac1\beta$ at $1$ and $1-\tfrac1\beta$ at $-1$.
\end{remarks}

\begin{proof}
(a) was proved in~\cite[Prop~5.3]{KN}.

Let us prove (b). First of all note that $\calC$ is a real orthogonal matrix. Equation~\eqref{characteristicPoly}, Lemma~\ref{OP_lemma}(i), and $\alpha_{2n-1}=1$ implies $\det\calC = -1$. Therefore $\calC$ has eigenvalues $e^{\pm i\theta_1},\ldots,e^{\pm i\theta_{n-1}},1,-1$.

Next we note that there is a one-to-one map between $2n\times 2n$ CMV matrices of the form $\calC(\alpha_0,\ldots,\alpha_{2n-2},1)$ with all $\alpha_j\in(-1,1)$ and spectral measures of the form~\eqref{formSpectralMeasureO2} with~\eqref{domainO3},~\eqref{domainO4}, and distinct $\theta_j$'s. Indeed, all Verblunsky coefficients are real if and only if the spectral measure is symmetric. And as we just saw, in this case $\alpha_{2n-1}=1$ is equivalent to $\det\calC = -1$, which for orthogonal matrices with simple spectrum is equivalent to having an eigenvalue at each $1$ and $-1$.

This map therefore induces a one-to-one mapping between $(\alpha_0,\ldots,\alpha_{2n-2})\in (-1,1)^{2n-1}$ and $(\theta_1,\ldots,\theta_{n-1},\mu_1,\ldots,\mu_n)$ in \eqref{domainO3},~\eqref{domainO4} with $\theta_1<\ldots<\theta_n$. 
So we can speak of the corresponding Jacobian. In fact, combining Propositions~\ref{thmOCMV} and~\ref{thmOSpectr}(b), we  obtain its value:
\begin{equation*}
\left|\det\frac{\partial(\alpha_0,\ldots,\alpha_{2n-2})}{\partial(\theta_1,\ldots,\theta_{n-1},\mu_1,\ldots,\mu_n)}\right|
= 2^{n-1}  
  \frac{\left| \Delta(2\cos\theta_1,\ldots,2\cos\theta_{n-1}) \right|^{2} \prod_{j=1}^{n-1} \sin^{2}\theta_j }{\sqrt{\mu_n \mu_{n+1}}\prod_{k=0}^{2n-2} (1-\alpha^2_k)^{\tfrac{2n-1-k}{2}-1}  }.
\end{equation*}

Using the form~\eqref{formSpectralMeasureO2} of the spectral measure, Lemma~\ref{OP_lemma}(v) gives
\begin{align*}
\prod_{j=0}^{2n-2} & (1- |\alpha_j|^2)^{(2n-j-1)/2} \\
&= \sqrt{\mu_n \mu_{n+1}}\prod_{j=1}^{n-1} \tfrac{\mu_j}{2} \left| \Delta\left( e^{i\theta_1},e^{-i\theta_1},\ldots,e^{i\theta_{n-1}},e^{-i\theta_{n-1}},1,-1  \right)\right| \\
&=2^{2n-1} \sqrt{\mu_n \mu_{n+1}}\prod_{j=1}^{n-1} \mu_j \prod_{j=1}^{n-1} |\sin \theta_j|^3 \left| \Delta\left( 2\cos\theta_1,\ldots,2\cos\theta_{n-1} \right) \right|^2.
\end{align*}

Now, starting with the distribution~\eqref{generalSOCMV2}, and using the change of variables along with the last computation produces the promised distribution of the spectral measure. The normalization constant $L_n(\beta)$ is a well-known integral (see, e.g.,~\cite[Lem~A.4]{KN}).

Parts (c) and (d) are similar.
\end{proof}

\begin{proposition}\label{thmUSpCMV}
The $2\times2$ block CMV-fication of the $\bbUSp(n)$ ensemble produces the block CMV matrix $\calC(\alpha_0,\ldots,\alpha_{n-1})$ with independent $2\times 2$ Verblunsky coefficients $\alpha_j$  distributed as
\begin{equation*}
\alpha_k \sim \Upsilon(4n-4k-1), \quad 0\le k \le n-1.
\end{equation*}
\end{proposition}
\begin{remark}
Scalar CMV-fication of $\bbUSp(n)$ is not natural and does not lead to statistically independent Verblunsky coefficients.
\end{remark}
\begin{proof}
With the help of our Proposition~\ref{columns}(c), the proof for $\bbUSp(n)$ goes along the same lines as the proof for $\bbU(n)$ in~\cite[Prop 3.3]{KN}.
\end{proof}


\section{Matrix Models for Truncations}\label{sModels2}

In order to deal with the truncations, 
we will need the following lemma.

\begin{lemma}\label{truncations}
    Consider a unitary CMV matrix $\calC(\alpha_0,\ldots,\alpha_n)$ with $\alpha_n\in\partial\bbD$.
    \begin{itemize}
    \item If $n$ is even then the matrix obtained by removing the first row and column is equal to
    \begin{equation*}
      S^\dagger \calC\left(-\bar\alpha_{n-1}{\alpha}_n,-\bar\alpha_{n-2}{\alpha}_n,\ldots,-\bar\alpha_0{\alpha}_n\right)^T S
    \end{equation*}
    for some unitary $S$.
    \item If $n$ is odd then the matrix obtained by removing the first row and column is equal to
    \begin{equation*}
      S^\dagger \calC\left(-\bar{\alpha}_{n-1} \alpha_n,-\bar{\alpha}_{n-2} \alpha_n,\ldots,-\bar{\alpha}_0 \alpha_n\right) S
    \end{equation*}
    for some unitary $S$.
    \end{itemize}
\end{lemma}
\begin{proof}
    Suppose first that $n$ is even. Then $\calC(\alpha_0,\ldots,\alpha_n)$ with the first row and column removed is equal to $\hat{\calL} \hat{\calM}$, where
\begin{equation*}
\begin{aligned}
\hat{\calL} &=\diag\left( \left[-\alpha_0\right],\Xi[\alpha_2],\ldots,\Xi[\alpha_{n-2}],\left[\bar{\alpha}_n\right]  \right), \\
\hat{\calM} &=\diag\left( \Xi[\alpha_1],\Xi[\alpha_3],\ldots,\Xi[\alpha_{n-1}] \right),
\end{aligned}
\end{equation*}
where $\Xi[\alpha_k]$ is defined as in \eqref{CMV1}.

Now consider $(Q^\dagger P_2  \hat{\calM}^T P_1^\dagger Q)( Q^\dagger P_1 \hat{\calL}^T P_2^\dagger Q ) $,
where
\begin{equation*}
Q=
\begin{bmatrix}
0 & \cdots & 0 & 1 \\
0 & \cdots & 1 & 0 \\
\vdots & \ddots & \vdots & \vdots \\
1 & \cdots & 0 & 0
\end{bmatrix},
\end{equation*}
and $P_1=\diag(\alpha_n,1,\alpha_n,1,\ldots)$, $P_2=\diag(1,\alpha_n,1,\alpha_n,\ldots)$.

Direct calculations yield
\begin{equation*}
\begin{aligned}
Q^\dagger P_2  \hat{\calM}^T P_1^\dagger Q &= \diag(\Xi[-\bar\alpha_{n-1}{\alpha}_n],\ldots,\Xi[-\bar\alpha_{1}{\alpha}_n]), \\
Q^\dagger P_1 \hat{\calL}^T P_2^\dagger Q &= \diag([1],\Xi[-\bar\alpha_{n-2}{\alpha}_n],\ldots,\Xi[-\bar\alpha_{2}{\alpha}_n],[-\overline{\bar{\alpha}_0 {\alpha}_n}])
\end{aligned}
\end{equation*}
This shows  $Q^\dagger P_2  (\hat{\calL} \hat{\calM})^T P_2^\dagger Q  = (Q^\dagger P_2  \hat{\calM}^T P_1^\dagger Q)( Q^\dagger P_1 \hat{\calL}^T P_2^\dagger Q )$ is precisely equal to $\calC(-\bar\alpha_{n-1}{\alpha}_n,-\bar\alpha_{n-2}{\alpha}_n,\ldots,-\bar\alpha_{0}{\alpha}_n)$, cf. Definition \ref{CMV}.

Similarly, if $n$ is odd, we can see that $(Q^\dagger P_1^\dagger \hat{\calL} P_2 Q ) (Q^\dagger P_2^\dagger  \hat{\calM} P_1 Q)$ is equal to $\calC\left(-\bar{\alpha}_{n-1} \alpha_n,-\bar{\alpha}_{n-2} \alpha_n,\ldots,-\bar{\alpha}_0 \alpha_n\right)$.
\end{proof}

With this lemma and the results of Section~\ref{sModels} in hand, we immediately obtain matrix models for the truncations of unitary ensembles.

Indeed, choose a matrix $U$ from one of the ``scalar'' ensembles from Section~\ref{sSpectral} ($CUE, COE, \bbO, \bbSO, \bbO\setminus\bbSO$), and remove any row and the corresponding column.
By the invariance under the group action, it does not matter which row and column we delete. It will be convenient to delete the first ones. Then the truncated matrix $V$ is
$$
V=PUP^\dagger, \quad \mbox{ where }  P=\operatorname{proj}\left(\langle e_1\rangle^\perp \right).
$$

Recall now~\eqref{W} that $U=W\calC W^\dagger$ with $W e_1=W^\dagger e_1 = e_1$ which implies
$$
V=PW\calC W^\dagger P^\dagger = W P \calC P^\dagger W^\dagger.
$$
Therefore $V$ is unitarily equivalent to the CMV matrix form $\calC$ of $U$ but with the first row and column removed. Thus the results of the next proposition follow from the CMV models in Section~\ref{sModels}, the previous lemma, and the rotational (or real even) symmetry of the Verblunsky coefficients.

For quaternionic ensembles $CSE$ and $\bbUSp$ one needs to modify all of the arguments in the obvious way to accommodate quaternionic rows and columns.

\begin{proposition}\label{thmTruncModels}\leavevmode
\begin{itemize}
\item[(a)]
The eigenvalue distribution of a $COE(n+1)$ or $CUE(n+1)$ matrix with one row and column removed coincides with the eigenvalue distribution of $\calC(\alpha_0,\ldots,\alpha_{n-1})$ with independent Verblunsky coefficients
\begin{equation}\label{truncDyson}
\alpha_k \sim \Theta(\beta(k+1)+1), \quad 0\le k \le n-1,
\end{equation}
where $\beta=1$ for $COE$ and $\beta=2$ for $CUE$.

The eigenvalue distribution of a $CSE(n+1)$ matrix with one (quaternionic) row and column removed coincides with the eigenvalue distribution of $\calC(\alpha_0 I_2,\ldots,\alpha_{n-1} I_2)$ where $\alpha_k$ are independent and distributed according to~\eqref{truncDyson} with $\beta=4$.

\item[(b)] The eigenvalue distribution of a $\bbO(n+1)$, or $\bbSO(n+1)$, or $\bbO(n+1)\setminus\bbSO(n+1)$  matrix with one row and column removed coincides with the distribution of $\calC(\alpha_0,\ldots,\alpha_{n-1})$ with independent Verblunsky coefficients
\begin{equation*}\label{truncO}
\alpha_k \sim B\left( \tfrac{k+1}{2}, \tfrac{k+1}{2}\right), \quad 0\le k \le n-1.
\end{equation*}

\item[(c)] The eigenvalue distribution of a $\bbUSp(n+1)$ matrix with one (quaternionic) row and column removed coincides with the distribution of $\calC(\alpha_0,\ldots,\alpha_{n-1})$ with independent $2\times2$ matrix Verblunsky coefficients
\begin{equation*}\label{truncUSp}
\alpha_k \sim \Upsilon\left( 4k+7 \right), \quad 0\le k \le n-1.
\end{equation*}
\end{itemize}
\end{proposition}
\begin{remarks}
1. From the aesthetic considerations, one might prefer to look at the above distributions as $B(\tfrac{1(k+2)-1}{2},\tfrac{1(k+2)-1}{2})$, $\Theta(2(k+2)-1)$, $\Upsilon(4(k+2)-1)$ for the truncations of the compact groups $\bbO$, $\bbU$, and $\bbUSp$, respectively. Note $B$ has an extra factor of $\tfrac12$ due to the historical definition of the Beta function.

2. These CMV matrices are no longer unitary as the last coefficient is no longer unimodular (with probability 1), see the remark after Definition~\ref{CMV}.

3. These models are efficient for simulation purposes.
\end{remarks}

\section{Eigenvalue Distribution of the Truncations}\label{sEigenvalues}
\subsection{Truncations of  circular $\beta$-ensembles}\label{ssEigenvaluesDyson}

\begin{theorem}\label{thm1}
Choose Verblunsky coefficients $\alpha_k$ to be independent and $\Theta(\beta(k+1)+1)$-distributed $(0\le k \le n-1)$, that is, distributed in $\bbD$ with to the density
\begin{equation*}
\tfrac{\beta (k+1)}{2\pi} (1-|\alpha_k|^2)^{\frac{\beta}{2}(k+1)-1} \, d^2\alpha_{k}.
\end{equation*}

Then the eigenvalues $z_1,\ldots,z_n$ of the CMV matrix $\calC(\alpha_0,\ldots,\alpha_{n-1})$ are distributed in $\bbD^n$ with the joint law
\begin{equation}\label{eigenvalues}
\tfrac{\beta^n}{(2\pi)^n} \prod_{j,k=1}^n (1-z_j \bar{z}_k)^{\frac{\beta}{2}-1}  \prod_{j<k} |z_k-z_j|^2 \, d^2 z_1\ldots d^2 z_n.
\end{equation}
\end{theorem}
\begin{remarks}
1. For any $0<\beta<\infty$ these ensembles are unitarily equivalent to the circular $\beta$-ensembles with the first row and column removed.

2. Combining this with Proposition~\ref{thmTruncModels}, we obtain Theorem~\ref{introThm}.

3. As we will see in Section~\ref{sLoggas}, the law \eqref{eigenvalues} coincides with the Gibbs measure for a system of $n$ charges in the disk $\D$ where $\D$ and $\C\setminus\D$ are filled with dielectrics of suitably chosen susceptibility.
\end{remarks}


The proof is given in the end of this subsection after we establish a series of lemmas.

Given a sequence of Verblunsky coefficients $\alpha_0,\ldots,\alpha_{n-1}$, one may then form the associated sequence of orthogonal polynomials $\Phi_0,\ldots,\Phi_n$; see Appendix~\ref{sOPUC}. Recall from~\eqref{characteristicPoly} that the eigenvalues of $\calC(\alpha_0,\ldots,\alpha_{n-1})$ coincide with the zeros of $\Phi_n(z)$, which we shall denote by $z_1,\ldots,z_n$. Note that the mapping from the zeros $(z_1,\ldots,z_n)\in\D^n$ to the Verblunsky coefficients $(\alpha_0,\ldots,\alpha_{n-1})\in\bbD^n$ is an $n$-fold branched covering. This follows from~\cite[Thm 1.7.5]{OPUC1}.  To compute the Jacobian of this change of variables, it will be useful to introduce an intermediate set of variables. Let $\kappa_l^{(k)}$ be the coefficients of the $k$-th orthogonal polynomial $\Phi_k$:
\begin{equation}\label{phiK}
\Phi_k(z)=z^k+\kappa_1^{(k)} z^{k-1}+\ldots+\kappa_{k-1}^{(k)} z+\kappa_k^{(k)}.
\end{equation}

The Jacobian of the mapping from roots of a polynomial to its coefficients is well known; we reproduce it here as Lemma~\ref{jacobian}(b).  Viewing complex numbers as
ordered pairs of real numbers, we have
\begin{equation}\label{DysonJacobian1}
  J_\bbR\left[\frac{\partial(\kappa^{(n)}_1,\ldots,\kappa^{(n)}_n)}{\partial(z_1,\ldots,z_n)}\right] = | \Delta(z_1,\ldots,z_n)|^2.
\end{equation}

\begin{lemma}\label{lemma1}
The Jacobian $J_\bbR\left[\frac{\partial(\kappa^{(n)}_1,\ldots,\kappa^{(n)}_n)}{\partial(\alpha_0,\ldots,\alpha_{n-1})}\right]$  is equal to $(-1)^n \prod_{k=0}^{n-1} (1-|\alpha_k|^2)^k$.
\end{lemma}
\begin{proof}[Proof of Lemma \ref{lemma1}]
From \eqref{recurrence} we can immediately see that
\begin{equation}\label{kappaRecurrence}
\kappa_j^{(n)} = \kappa_j^{(n-1)} - \bar{\alpha}_{n-1} \bar{\kappa}_{n-j}^{(n-1)}, \quad j=1,\ldots,n,
\end{equation}
where we assign $\kappa_n^{(n-1)}=0$ and $\bar{\kappa}_{0}^{(n-1)}=1$. Therefore applying Lemma~\ref{lemJacobian}, we obtain
\begin{multline}\label{kappaToAlpha}
J_\bbR\left[\tfrac{\partial(\kappa^{(n)}_1,\ldots,\kappa^{(n)}_n)}{\partial(\kappa_1^{(n-1)},\ldots,\kappa_{n-1}^{(n-1)},\alpha_{n-1})}\right]
=
\\
\det
\begin{bmatrix}
1 & 0 &  \cdots  & 0 & -\alpha_{n-1} & 0 & 0 \\
0 & 1 & \cdots & -\bar{\alpha}_{n-1} & 0 & 0 & 0 \\
\vdots & \vdots &  \ddots & \vdots & \vdots & \vdots & \vdots \\
0 & -\alpha_{n-1} &  \cdots & 1 & 0 & 0 & 0 \\
-\bar{\alpha}_{n-1} & 0  & \cdots  & 0 & 1 & 0 & 0 \\
0 & -\kappa_{n-1}^{(n-1)} &  \cdots & 0 & -\kappa_{1}^{(n-1)} & 0 & -1 \\
-\bar{\kappa}_{n-1}^{(n-1)} & 0 & \cdots & -\bar{\kappa}_{1}^{(n-1)} & 0 &  -1 & 0
\end{bmatrix}.
\end{multline}
Expanding this determinant along the last two columns and then rearranging the remaining rows and columns we obtain a block diagonal matrix with $n-1$ blocks $$\left[
\begin{array}{cc}
1 & -\alpha_{n-1} \\
-\bar{\alpha}_{n-1} & 1
\end{array}
\right]
$$
on the diagonal. Therefore~\eqref{kappaToAlpha} simplifies to $ -(1-|\alpha_{n-1}|^2)^{n-1}$ (minus sign comes from expanding along the last two columns).

The lemma now follows by induction.
\end{proof}

\begin{lemma}\label{lemma2}
The Jacobian $J_\bbR\left[ \frac{\partial(\alpha_0,\ldots,\alpha_{n-1})}{\partial(z_1,\ldots,z_n)} \right]$  is equal to $\frac{| \Delta(z_1,\ldots,z_n)|^2}{(-1)^n \prod_{k=0}^{n-1} (1-|\alpha_k|^2)^k}$.
\end{lemma}
\begin{proof}
This is immediate from~\eqref{DysonJacobian1} and Lemma \ref{lemma1}.
\end{proof}

Now we can prove Theorem \ref{thm1}.
\begin{proof}[Proof of Theorem \ref{thm1}]
The joint distribution of $\alpha_0,\ldots,\alpha_{n-1}$ is
\begin{equation*}
\tfrac{\beta^n n!}{(2\pi)^n} \prod_{k=0}^{n-1} (1-|\alpha_k|^2)^{\frac{\beta}{2}(k+1)-1} \, d^2\alpha_0 \ldots d^2\alpha_{n-1}.
\end{equation*}
Using Lemmas \ref{lemma2} and \ref{OP_lemma}(vi) we see that this induces the following measure on $z_1,\ldots,z_n$:
\begin{equation*}
\begin{aligned}
\tfrac{1}{n!} \tfrac{\beta^n n!}{(2\pi)^n} &\left( \prod_{k=0}^{n-1} (1-|\alpha_k|^2)^{\frac{\beta}{2}(k+1)-1} \right)
\tfrac{| \Delta(z_1,\ldots,z_n)|^2}{\prod_{k=0}^{n-1} (1-|\alpha_k|^2)^k}
\, d^2 z_1 \ldots d^2 z_n  \\
&= \tfrac{\beta^n}{(2\pi)^n}  | \Delta(z_1,\ldots,z_n)|^2 \prod_{k=0}^{n-1} (1-|\alpha_k|^2)^{(\frac{\beta}{2}-1)(k+1)}
\, d^2 z_1 \ldots d^2 z_n \\
&= \tfrac{\beta^n}{(2\pi)^n}  | \Delta(z_1,\ldots,z_n)|^2 \prod_{j,k=1}^n (1-z_j \bar{z}_k)^{\frac{\beta}{2}-1}
\, d^2 z_1 \ldots d^2 z_n.
\end{aligned}
\end{equation*}
Note that we removed the ordering of $z_j$'s which resulted in the $\tfrac{1}{n!}$ factor.
\end{proof}

\subsection{Truncation of the orthogonal $\beta$-ensemble}\label{ssEigenvaluesOrtho}

Truncations of the orthogonal group, and more generally of the orthogonal $\beta$-ensemble, have real eigenvalues with non-zero probability.  Correspondingly, we decompose the set of possible systems of eigenvalues as
$$
D:=\bigcup_{L+2M = n,L\ge0, M\ge 0} D_{L,M},
$$
where
\begin{multline}\label{setD}
D_{L,M} := \left\{ (x_1,\ldots,x_L,x_{L+1}\pm iy_{L+1},\ldots,x_{L+M}\pm iy_{L+M}) \in \bbD^{L+2M}: \right.
\\ \left. -1< x_1< x_2<\ldots< x_L<1, \right.
\left. -1<x_{L+1}<\ldots < x_{L+M} <1, \right. \\
\left. y_{L+1}>0,\ldots, y_{L+M}>0 \right\}.
\end{multline}

Assume we are given a function $f:\bbC^n\to\bbC^n$ that is non-negative on $D$ and invariant under the permutations of its $n$ arguments.  Let us define $f(z_1,\ldots,z_n) |d z_1\wedge dz_2 \wedge\ldots\wedge d z_n|$ to be the distribution on $D\subset \bbD^n$ determined by
\begin{multline*}
\int_{\bbD^n} f(z_1,\ldots,z_n) |d z_1\wedge dz_2 \wedge\ldots\wedge d z_n| \\ := \sum_{L+2M=n} 2^M \int_{D_{L,M}} f(x_1,\ldots,x_L,x_{L+1}\pm iy_{L+1},\ldots,x_{L+M}\pm iy_{L+M}) \\
\times dx_1\ldots dx_{L+M} dy_{L+1}\ldots dy_{L+M}.
\end{multline*}

Informally one should think of
each of the summands on the right hand side as the distribution conditional on the event that the matrix has $L$ real eigenvalues and $M$ complex-conjugate pairs with $L+2M=n$.  Each of these distributions is obtained by plugging into the left-hand side $dz_k=dx_k+i dy_k$, taking the wedge product while keeping track  which of the eigenvalues are real or complex, and then taking the absolute value of the constant. Thus for each real eigenvalue $z_k$, $dz_k$ becomes just the real Lebesgue measure $dx_k$ on the real line, and for a complex-conjugate pair $z_k$, $\bar{z}_k$, we get $\left|(dx_k+idy_k)\wedge(dx_k-idy_k) \right| = 2 dx_k dy_k$. This corresponds to the factor $2^M$ on the right-hand side of the definition.


More formal and careful discussion of these measures can be found, e.g., in the Borodin--Sinclair paper~\cite[Sects 2--3]{BorSin}.

\begin{theorem}\label{thm_beta}
Given $\beta>0$, $a,b>-1$, let the coefficients $\alpha_k$, $0\le k \le n-1$, be distributed as follows:
\begin{equation}\label{jointAlphaOrtho}
\alpha_k\sim \left\{
\begin{matrix}
B\left(\tfrac{\beta k}{4}+a+1,\tfrac{\beta k}{4}+b+1\right), & \mbox{if } k \mbox{ is even,} \\
B\left(\tfrac{\beta (k-1)}{4}+a+b+2,\tfrac{\beta (k+1)}{4}\right), & \mbox{if } k \mbox{ is odd.}
\end{matrix}
\right.
\end{equation}
Then eigenvalues of the CMV matrix $\calC(\alpha_0,\ldots,\alpha_{n-1})$ are distributed in $\bbD$ with the joint law
\begin{equation*}\label{eigenvaluesOrtho}
\tfrac{1}{P_n} \prod_{j,k=1}^n (1-z_j \bar{z}_k)^{\frac{\beta}{4}-\frac12}   \prod_{j=1}^n (1-z_j)^{a+\frac12-\frac\beta4} (1+z_j)^{b+\frac12-\frac\beta4} \prod_{j<k} |z_k-z_j| \, |d z_1\wedge dz_2 \wedge\ldots\wedge d z_n|,
\end{equation*}
where
\begin{equation}\label{P_n}
P_n = 2^{n(a+b+1) + \frac{\beta}{4}n(n-1)} \prod_{j=0}^{\lfloor \frac{n-1}{2} \rfloor} \frac{\Gamma(a+1+\frac{\beta}{2}j)\Gamma(b+1+\frac{\beta}{2}j) }{\Gamma(a+b+2+\frac{\beta}{2}(\lfloor\frac{n}{2}\rfloor+j))} \prod_{j=1}^{\lfloor \frac{n}{2} \rfloor} \Gamma(\tfrac{\beta}{2}j).
\end{equation}
\end{theorem}
\begin{remarks}
1. From our previous discussion it is evident that if $n$ is odd, then such a random matrix is unitarily equivalent to a $(n+1)\times (n+1)$ matrix from the orthogonal $\beta$-ensemble in Proposition~\ref{thmRealOrtho}(a) with the first row and column removed (for any $0<\beta<\infty$ and $a,b>-1$). If $a=b=-1+\beta/4$, then it is also unitarily equivalent to a $(n+1)\times (n+1)$ matrix from the ensemble in Proposition~\ref{thmRealOrtho}(b) with the first row and column removed (for any $0<\beta<\infty$). If $n$ is even and $a=b=-1+\beta/4$, then it is unitarily equivalent to a $(n+1)\times (n+1)$ matrix from the ensembles in Proposition~\ref{thmRealOrtho}(c) and (d) with the first rows and columns removed.

2. When $\beta=2, a=b=-1/2$, such a random matrix is unitarily equivalent to a matrix from $\bbO(n+1)$, or $\bbSO(n+1)$, or $\bbO(n+1)\setminus\bbSO(n+1)$ with any row and column removed, cf. Proposition~\ref{thmTruncModels}(b). 
\end{remarks}

We follow the same strategy as for the circular ensembles. As before, let~\eqref{phiK}. In the next lemma we compute the Jacobian determinant of the transformation from $z_j$'s to $\kappa_j^{(n)}$'s on each of the sets $D_{L,M}$.

\begin{lemma}
Suppose $z_1,\ldots, z_n$ are in $D_{L,M}$, $L+2M=n$, see \eqref{setD}. On this set we have that the mapping $(z_1,\ldots,z_n) \mapsto (\kappa^{(n)}_1,\ldots,\kappa^{(n)}_n)$ is injective, and
\begin{equation*}
\left| J\left[ \frac{\partial(\kappa^{(n)}_1,\ldots,\kappa^{(n)}_n)}{\partial(x_1,\ldots,x_L,x_{L+1},y_{L+1},x_{L+2},y_{L+2},\ldots,x_{L+M},y_{L+M})} \right] \right|
  = 2^M \Delta(z_1,\ldots,z_n).
\end{equation*}
\end{lemma}
\begin{remark}
In other words, the above lemma says that  the following change of variables is legitimate
$$
f(\ldots) \, \,  d\kappa^{(n)}_1 \ldots d\kappa^{(n)}_n = f(\ldots)\, |\Delta(z_1,\ldots,z_n)|\,\,  |d z_1\wedge dz_2 \wedge\ldots\wedge d z_n|.
$$
\end{remark}
\begin{proof}
  Note that the above change of variables can be viewed as a composition of two changes of variables:
\begin{equation*}
  \begin{aligned}
    (\kappa^{(n)}_1,\ldots,\kappa^{(n)}_n)&\mapsto (x_1,\ldots,x_L,x_{L+1}\pm iy_{L+1},\ldots,x_{L+M}\pm iy_{L+M}) \\ &\mapsto (x_1,\ldots,x_L,x_{L+1},y_{L+1},x_{L+2},y_{L+2},\ldots,x_{L+M},y_{L+M})
  \end{aligned}
\end{equation*}
But the (complex) Jacobian determinant for the first map is $\Delta(z_1,\ldots,z_n)$ by Lemma \ref{jacobian}(a), and the (complex) Jacobian for the second map is easily seen to be $2^M$ in the absolute value.
\end{proof}

\begin{lemma}
The Jacobian determinant of the $\bbR^n\to\bbR^n$ change of variables $(\alpha_j)_{j=0}^{n-1} \mapsto (\kappa_j^{(n)})_{j=1}^n$ is
\begin{equation*}
 J\left[ \frac{\partial(\kappa^{(n)}_1,\ldots,\kappa^{(n)}_n)}{\partial(\alpha_0,\ldots,\alpha_{n-1})} \right] =
  (-1)^n \prod_{\substack{0\le k\le n-1 \\ k\mbox{\footnotesize{ even}}}} (1-\alpha_k^2)^{k/2}
  \prod_{\substack{0\le k\le n-1 \\ k\mbox{\footnotesize{ odd}}}} (1-\alpha_k)(1-\alpha_k^2)^{(k-1)/2}.
\end{equation*}
\end{lemma}
\begin{proof}
Using the recurrence~\eqref{kappaRecurrence} it is easy to see that
\begin{multline}\label{kappaToAlpha2}
J\left[\tfrac{\partial(\kappa^{(n)}_1,\ldots,\kappa^{(n)}_n)}{\partial(\kappa_1^{(n-1)},\ldots,\kappa_{n-1}^{(n-1)},\alpha_{n-1})}\right]
\\
=
\det
\begin{bmatrix}
1 & 0 &  \cdots  & 0 & -\alpha_{n-1} & 0  \\
0 & 1 & \cdots & -{\alpha}_{n-1} & 0 & 0  \\
\vdots & \vdots &  \ddots & \vdots & \vdots & \vdots  \\
0 & -\alpha_{n-1} &  \cdots & 1 & 0 & 0 \\
-{\alpha}_{n-1} & 0  & \cdots  & 0 & 1 & 0 \\
-\kappa_{n-1}^{(n-1)} & -\kappa_{n-2}^{(n-1)} &  \cdots & -\kappa_{2}^{(n-1)} & -\kappa_{1}^{(n-1)} & -1
\end{bmatrix}.
\end{multline}
If $n$ is even then the $(\tfrac{n}{2},\tfrac{n}{2})$-entry of the latter matrix is $1-\alpha_{n-1}$.
We stress that this is an $n\times n$ matrix, in contrast to~\eqref{kappaToAlpha} which was $2n\times 2n$. Expansion along the last column allows to compute~\eqref{kappaToAlpha2}, and then induction finishes the proof.
\end{proof}

Now we are ready to prove Theorem \ref{thm_beta}.

\begin{proof}[Proof of Theorem \ref{thm_beta}]
Starting with the joint probability density~\eqref{jointAlphaOrtho} of $\alpha_k$'s and performing two changes of variables from the previous two lemmas, we obtain
\begin{multline*}
\tfrac{1}{P_n} \prod_{\substack{0\le k\le n-1 \\ k\mbox{\footnotesize{ even}}}} (1-\alpha_k)^{\frac{k\beta}{4}+a-\frac{k}{2}} (1+\alpha_k)^{\frac{k\beta}{4}+b-\frac{k}{2}}
\\
\times
  \prod_{\substack{0\le k\le n-1 \\ k\mbox{\footnotesize{ odd}}}} (1-\alpha_k)^{\frac{(k-1)\beta}{4}+a+b+1 - \frac{k+1}{2}} (1+\alpha_k)^{\frac{(k+1)\beta}{4}-1 -\frac{k-1}{2}}
 \\
 \times |\Delta(z_1,\ldots,z_n)|\,  |d z_1\wedge dz_2 \wedge\ldots\wedge d z_n|,
\end{multline*}
where the constant
\begin{multline*}
\tfrac{1}{P_n}  =\prod_{\substack{0\le k\le n-1 \\ k\mbox{\footnotesize{ even}}}} \frac{\Gamma(\tfrac{\beta}{2}k+a+b+2)}{2^{\frac{\beta k}{2}+a+b+1} \Gamma(\frac{\beta}{4}k+a+1)\Gamma(\frac{\beta}{4}k+b+1)}
\\
\times
\prod_{\substack{0\le k\le n-1 \\ k\mbox{\footnotesize{ odd}}}} \frac{\Gamma(\frac{\beta}{2}k+a+b+2)}{2^{\frac{\beta k}{2}+a+b+1} \Gamma(\frac{\beta}{4}(k-1)+a+b+2)\Gamma(\frac{\beta}{4}(k+1))}
\end{multline*}
can be easily seen to be equivalent to~\eqref{P_n}.

After some simplification, the distribution becomes
\begin{multline*}
\tfrac{1}{P_n}  \prod_{k=0}^{n-1} (1-\alpha_k^2)^{(k+1)(\frac\beta4-\frac12)} \prod_{k=0}^{n-1} (1-\alpha_k)^{a-\frac\beta4+\frac12}
\prod_{k=0}^{n-1} (1+(-1)^k \alpha_k)^{b-\frac\beta4+\frac12}
 \\
 \times
  |\Delta(z_1,\ldots,z_n)| \, |d z_1\wedge dz_2 \wedge\ldots\wedge d z_n|.
\end{multline*}
Now, an application of Lemma~\ref{OP_lemma}(i) and (vi) finishes the proof.
\end{proof}

%


\section{Matrix Models and Resonance Distribution for Systems with Non-ideal Coupling}\label{sNonperfect}
As argued by Fyodorov--Sommers in~\cite{FyoSom00} (see also~\cite{Fyo01,FyoSom}), the eigenvalues of the truncations represent resonances for the open systems for the special case of the ideal coupling. They proposed a more general model (for $CUE$) which for the 1 open channel case takes form
\begin{equation}\label{non-ideal}
V=U \operatorname{diag}\left(\sqrt{1-T_a},1,1,\ldots,1\right).
\end{equation}
Here $U$ is a  unitary random matrix, and $0\le T_a \le 1$ is the so-called transmission coefficient, which may be random, or deterministic. The case of $T_a=0$ corresponds to the closed system, and $T_a=1$ to the open system with the ideal coupling. Clearly the eigenvalues corresponding to $T_a=0$ and $T_a=1$ (ignoring the zero eigenvalue) coincide with the eigenvalues of the unperturbed unitary random matrix and of its truncation, respectively.

In fact, instead of the transmission coefficient it will be more convenient for us to work with the reflection coefficient
$
R_a:=\sqrt{1-T_a}.
$


Our method of reducing to the CMV form applies here as well. Indeed, let $U$ be $n \times n$ and its CMV form be $U=W\calC(\alpha_0,\ldots,\alpha_{n-1}) W^\dagger$ with $We_1=W^\dagger e_1 = e_1$. Then
$$
V = W \calC(\alpha_0,\ldots,\alpha_{n-1}) \operatorname{diag}\left(R_a,1,1,\ldots,1\right) W^\dagger.
$$
Using the idea in Lemma~\ref{truncations}, one then obtains that $V$ is  unitarily equivalent to
$$
\calC\left(-\bar\alpha_{n-2}{\alpha}_{n-1},-\bar\alpha_{n-3}{\alpha}_{n-1},\ldots,-\bar\alpha_0{\alpha}_{n-1}, R_a \alpha_{n-1}  \right)^T
$$
if $n$ is odd and to
$$
\calC\left(-\bar{\alpha}_{n-2} \alpha_{n-1},-\bar{\alpha}_{n-3} \alpha_{n-1},\ldots,-\bar{\alpha}_0 \alpha_{n-1}, R_a \alpha_{n-1}  \right)
$$
if $n$ is even. 

Combining this with the CMV models we developed earlier,  we immediately obtain the following proposition.


\begin{proposition}
Let $0\le R_a \le 1$ be independent of the random matrix $U$.
\begin{itemize}
\item[(a)]
The eigenvalue distribution of~\eqref{non-ideal} with $U\in COE(n), CUE(n), CSE(n)$ coincides with the eigenvalue distribution of $\calC(\alpha_0,\ldots,\alpha_{n-2},R_a e^{i\phi})$ with independent coefficients
\begin{equation}\label{non-idealDyson}
\alpha_k \sim \Theta(\beta(k+1)+1), \quad 0\le k \le n-2, \quad e^{i\phi} \in\Theta(1),
\end{equation}
where $\beta=1, 2, 4$ for $COE, CUE, CSE$, respectively.


\item[(b)] The eigenvalue distribution of~\eqref{non-ideal} with $U\in\bbO(n)$, or $\bbSO(n)$, or $\bbO(n)\setminus\bbSO(n)$ coincides with the distribution of $\calC(\alpha_0,\ldots,\alpha_{n-2},R_a \sigma)$ with independent coefficients
\begin{equation*}\label{non-idealOrtho}
\alpha_k \sim B\left( \tfrac{k+1}{2}, \tfrac{k+1}{2}\right), \quad 0\le k \le n-2,
\end{equation*}
and $\sigma\in B(0,0)$ for $\bbO(n)$; $\sigma=-1$ for $\bbSO(2k)$ or $\bbO(2k+1)\setminus \bbSO(2k+1)$; and $\sigma=1$ for $\bbSO(2k+1)$ or $\bbO(2k)\setminus\bbSO(2k)$.

\item[(c)] The eigenvalue distribution of a $\bbUSp(n)$ matrix with one row and column removed coincides with the distribution of $\calC(\alpha_0,\ldots,\alpha_{n-2},R_a q)$ with independent $2\times2$ matrix coefficients
\begin{equation*}\label{non-idealUSp}
\alpha_k \sim \Upsilon\left( 4k+7 \right), \quad 0\le k \le n-2, \quad q\in\Upsilon(3).
\end{equation*}
\end{itemize}
\end{proposition}
\begin{remarks}
1. As usual, $CSE$ has double multiplicity at each of the eigenvalues of~\eqref{non-idealDyson}.

2. When $R_a=0$ these CMV models coincide with the models in Section~\ref{sModels2} (with $n$ instead of $n+1$ that we had earlier), and when $R_a = 1$ they become the (reversed) models from Section~\ref{sModels} (the reversed models however do not have the spectral measures preserved but only the eigenvalues, so are less natural).
\end{remarks}

For (a) and (b) we can compute the eigenvalue distributions.


\begin{proposition}
(a) Let $0<\beta<\infty$ and
\begin{equation*}
\alpha_k \sim \Theta(\beta(k+1)+1), \quad 0\le k \le n-2, \quad e^{i\phi} \in\Theta(1).
\end{equation*}
Assume the reflection coefficient $R_a$ is independent of $\alpha_0,\ldots,\alpha_{n-2},\phi$ with a distribution on $(0,1)$ of the form
\begin{equation}\label{reflection}
F(R_a) \, \beta n (1-R_a^2)^{\frac{\beta}{2}n-1 } R_a dR_a,
\end{equation}
where $F$ is any function that makes~\eqref{reflection} a probability distribution.
Then the eigenvalues $z_1,\ldots,z_n$ of $\calC(\alpha_0,\ldots,\alpha_{n-2},R_a e^{i\phi})$  are distributed in $\bbD^n$ with the joint distribution
\begin{equation}\label{non-idealEigenvalues}
\tfrac{\beta^n}{(2\pi)^n} \prod_{j,k=1}^n (1-z_j \bar{z}_k)^{\frac{\beta}{2}-1}  \prod_{j<k} |z_k-z_j|^2 \, F\Bigg( \Big| \prod_{j=1}^n z_j \Big|  \Bigg)  \, d^2 z_1\ldots d^2 z_n.
\end{equation}

(b) Let $0<\beta<\infty$ and
\begin{equation*}
\alpha_k\sim B\left(\tfrac{\beta (k+1)}{4},\tfrac{\beta (k+1)}{4}\right), \quad 0\le k \le n-2, \quad  \sigma\sim B(0,0).
\end{equation*}
Assume the reflection coefficient $R_a$ is independent of $\alpha_0,\ldots,\alpha_{n-2},\sigma$ with a distribution on $(0,1)$ of the form
\begin{equation}\label{reflection2}
G(R_a) \, \tfrac{2\Gamma(\frac{\beta}{4}n+\frac12)}{\sqrt{\pi}\Gamma(\frac{\beta}{4}n)} (1-R_a^2)^{\frac{\beta}{4}n-1 } dR_a,
\end{equation}
where $G$ is any function that makes~\eqref{reflection2} a probability distribution.
Then the eigenvalues $z_1,\ldots,z_n$ of $\calC(\alpha_0,\ldots,\alpha_{n-2},R_a \sigma)$  are distributed in $\bbD^n$ with the joint distribution
\begin{equation}\label{non-idealEigenvalues2}
\tfrac{1}{P_n} \prod_{j,k=1}^n (1-z_j \bar{z}_k)^{\frac{\beta}{4}-\frac12}   \prod_{j=1}^n \frac{1}{\sqrt{1-z^2_j}} \prod_{j<k} |z_k-z_j| \, G\Bigg( \Big| \prod_{j=1}^n z_j \Big| \Bigg) \, |d z_1\wedge dz_2 \wedge\ldots\wedge d z_n|,
\end{equation}
where $P_n$ is~\eqref{P_n}.
\end{proposition}
\begin{remarks}
1. When $\beta=1,2,4$ in (a),~\eqref{non-idealEigenvalues} is the eigenvalue distribution of~\eqref{non-ideal} with $U\in COE(n),CUE(n),CSE(n)$, respectively. Also, for any $0<\beta<\infty$,~\eqref{non-idealEigenvalues} is the  eigenvalue distribution of~\eqref{non-ideal} with $U$ a matrix from the circular $\beta$-ensemble.

2. When $\beta =2$ in (b),~\eqref{non-idealEigenvalues2} is the eigenvalue distribution of~\eqref{non-ideal} with $U\in \bbO(n)$.  The analogue of Theorem~\ref{thm_beta} for general $a,b>-1$ can also be stated. We chose $a=b=-1+\beta/4$ to simplify the presentation.

3. When $F\equiv 1$ in (a) (or $G\equiv 1$ in (b)), we obtain that a system under the non-ideal coupling with such a reflection coefficient $R_a$ behaves as a system one size larger with ideal coupling. One should also recognize that~\eqref{reflection} with $F\equiv 1$ and ~\eqref{reflection2} with $G\equiv 1$ coincide with the distribution of $|\alpha_{n-1}|$ in Theorem~\ref{thm1} and Theorem~\ref{thm_beta}, respectively. Ultimately, this is a manifestation of the statistical independence of the Verblunsky coefficients.
\end{remarks}

\begin{proof}
The result is obvious from the arguments in Section~\ref{sEigenvalues}. The only additional fact we need to use is that $R_a=|\alpha_{n-1}| = \left| \prod_{j=1}^n z_j \right|$, see Lemma~\ref{OP_lemma}(i).
\end{proof}

\section{Log-gas Interpretation}\label{sLoggas}

In this section, we present an interpretation of the measure \eqref{eigenvalues} as the Gibbs measure for a configuration of charges in a certain geometry of dielectrics.  What follows was very much inspired by~\cite[Sect~15.8--15.9]{bForrester}.

Consider a system of $n$  particles of equal charge $+1$ that are confined to lie in the unit disk $\bbD$ of the complex plane. The unit disk $\bbD$ is filled with a homogeneous (linear) dielectric of permittivity $\epsilon_1$, while the rest of space $\bbC\setminus\bbD$ is filled with a homogeneous (linear) dielectric of permittivity $\epsilon_2$.   As usual, one may regard two-dimensional point charges as representing parallel lines of charge in a three-dimensional setting.

Recall (cf. \cite[\S7]{MR0121049}) that the electric potential due to a unit point charge at $z_0$ in the presence of dielectric media is the solution to
$- \nabla_z\cdot [\epsilon(z)\nabla_z V(z)] = \delta(z-z_0)$.  For the configuration described above, and $z_0\in\D$ the solution is
$$
V(z|z_0)=\begin{cases}
	-\tfrac{1}{2\pi\epsilon_1}\log |z-z_0| - \tfrac{1}{2\pi\epsilon_1} \tfrac{\epsilon_1-\epsilon_2}{\epsilon_1+\epsilon_2} \log|1-z \bar{z}_0| &\quad z\in \D, \\
		-\tfrac{1}{\pi(\epsilon_2+\epsilon_1)}\log |z-z_0| &\quad z\in \C\setminus\D.
	\end{cases}
$$
From this, one finds (cf. Problems~3 and~4 in \cite[\S7]{MR0121049}) that the polarization of the dielectrics exerts a force on the charge at $z_0\in\D$ of the form
$$
F = - \tfrac{1}{2\pi\epsilon_1} \tfrac{\epsilon_1-\epsilon_2}{\epsilon_1+\epsilon_2} \tfrac{z_0}{1-|z_0|^2}, \text{ which has potential } W(z_0) = - \tfrac{1}{4\pi\epsilon_1} \tfrac{\epsilon_1-\epsilon_2}{\epsilon_1+\epsilon_2} \log[1-|z_0|^2].
$$
This force pushes the charge toward the origin when $\epsilon_1>\epsilon_2>0$ and toward the boundary in the case $\epsilon_2>\epsilon_1>0$.

The above calculations reveal that the total energy of a system of $n$ charges located at the points $\{z_j\}_{j=1}^n \subseteq \D$ is then
$$
H = \sum_j W(z_j) + \tfrac12 \sum_{j\neq k} V(z_k|z_j),
$$
and consequently, the associated Gibbs measure at temperature $T$ has density
\begin{equation}\label{E:DielectricGibbs}
e^{-H/(k_BT)} = \prod_{j,k=1}^n (1-z_j \bar{z}_k)^{\frac{\alpha \gamma}{2}}  \prod_{j<k} |z_k-z_j|^\gamma
\end{equation}
with exponents
$$
\alpha = \tfrac{\epsilon_1-\epsilon_2}{\epsilon_1+\epsilon_2}
	\qtq{ and } \gamma = \tfrac{1}{2\pi\epsilon_1k_B T} .
$$

Note that any combination of $\gamma>0$ and $\alpha\in\R$ is within the realm of physical possibility.  In the traditional setting (as elaborated in \cite{MR0121049}, for example) the ratio $\epsilon_2/\epsilon_1$ can lie anywhere in the interval $[0,\infty]$, with the case $\epsilon_2/\epsilon_1=\infty$, for example, being realizable by using a conductor in the exterior region $\C\setminus\D$ and vacuum in $\D$.  On the other hand, the physical consistency of materials with negative (or zero) permittivity has been known for some time.  In recent years, such materials have even been manufactured; see, for example, \cite{Eng} for an introduction to the literature on this.

Comparing \eqref{eigenvalues} and \eqref{E:DielectricGibbs}, we see that the distribution of eigenvalues for truncated Dyson models coincides with that of a gas of charged particles in the presence of dielectrics; specifically, one chooses $\gamma=2$ and $\alpha =\frac\beta2-1$.

\section{Symmetric CMV Matrices}\label{sSymmetricCMV}

As we mentioned in Subsection~\ref{ssCMV}, in the case of a system with time-reversal invariance ($COE$ and $CSE$) reducing its scattering matrix $U$ to the CMV form is not natural since this breaks the initial symmetry. In this section we introduce a symmetric analogue of the CMV form which solves this issue. We will call these matrices~\eqref{SCMV1} ``symmetric CMV''.

In what follows let $\real T = \tfrac{1}{2}(T+T^\dagger)$ and $\imag T = \tfrac{1}{2i}(T-T^\dagger)$.

\begin{proposition}\label{thmSCMV}
Let $U$ be a unitary operator on $\ell^2(\bbZ_+)$ with cyclic vector $e_1$ and spectral measure~\eqref{spectralMeasure}. Then applying Gram--Schmidt orthonormalization procedure to $e_1, [\imag U ] e_1, [\real U ]e_1, [\imag (U^2)] e_1, [\real(U^2)] e_1, \ldots$ produces  a basis $\{s_k\}$ in which $U$ has the form
\begin{equation}\label{SCMV1}
\calS(\alpha_0,\alpha_1,\ldots):= \calN \calL \calN^T,
\end{equation}
where $\calL = \calL^T$ is as in~\eqref{CMV2}, and $\calN$ is the unitary matrix
\begin{align}
\label{SCMV2}
\calN &= \operatorname{diag}\left( [1],\Psi_1,\Psi_3,\ldots\right), \\
\label{SCMV3}
\Psi_k & =
\frac{1}{\sqrt{2(1-\real \alpha_k)}}
\left[
\begin{array}{cc}
i(1- \bar{\alpha}_k) & -i\rho_k \\
\rho_k & 1-\alpha_k
\end{array}
\right].
\end{align}

The matrix $\calS$ is $7$-diagonal and symmetric $\calS^T=\calS$.  In fact, $U= R \calS R^\dagger$ with $R^\dagger R = R R^\dagger = I$, $R e_1  = R^\dagger e_1 = e_1$. Moreover, if $U^T=U$, then $R$ is orthogonal.
\end{proposition}

\begin{proof}
By the spectral theorem (cf. \eqref{isometry}), we may conflate $U$ with the operator $f(z)\mapsto zf(z)$ in $L^2(d\mu)$ and $e_1$ with the constant function $1$.  In this way, we identify the basis $\{s_k\}_{k=0}^\infty$ with the functions formed by orthonormalizing the sequence $\{1,\sin\theta,\cos\theta,\sin 2\theta, \cos 2\theta, \ldots\}$ in $L^2(\mu)$.  Here and below we write $z=e^{i\theta}$.

Following the notation from~\cite[Section 4.2]{OPUC1}, let us write $\{\chi_k\}_{k=0}^\infty$ and $\{x_k\}_{k=0}^\infty$ for the orthonormal bases formed by orthonormalizing $\{1,z,z^{-1},z^2,z^{-2},\ldots\}$ and $\{1,z^{-1},z,z^{-2},z^2,\ldots\}$, respectively, in $L^2(d\mu)$.  Recall that the traditional CMV matrix $\calC=\calL\calM$ represents $U$ in the $\{\chi_k\}$ basis and that $\calM$ represents the change of basis from $\{\chi_k\}$ to $\{x_k\}$; specifically, $\calM_{jk} = \langle x_j, \chi_k\rangle$.

It is elementary to verify that $\calN^T\calN = \calM$ and consequently,
$$
\calN \calC \calN^{-1} = \calN \calL\calM \calN^{-1} = \calN \calL \calN^T.
$$
Thus, the key assertion to be proved is that $\calN$ represents the change of basis from $\{\chi_k\}$ to $\{s_k\}$.  To this end, let us first observe that
\begin{align*}
\chi_{2n-1}(z) = z^{1-n} \phi_{2n-1}(z) = \tfrac{1}{\rho_1\rho_2\cdots\rho_{2n-1}}\bigl[\rho_{2n-1} z^n + \cdots + 0 z^{-n} \bigr], \\
\chi_{2n}(z)   = z^{-n} \phi_{2n}^*(z) = \tfrac{1}{\rho_1\rho_2\cdots\rho_{2n-1}}\bigl[-\alpha_{2n-1} z^n + \cdots + z^{-n} \bigr].
\end{align*}
Here ellipsis indicates terms in the linear span of $\{z^{n-1},\ldots,z^{1-n}\}$; we will continue this convention for the remainder of the proof.

Proceeding directly from the definition of the Gram--Schmidt process, we have
$$
s_{2n-1}(z) = c \tfrac{z^n-z^{-n}}{2i} + \cdots + 0 \tfrac{z^n+z^{-n}}{2}
$$
for some positive constant $c$.  Correspondingly, there is a $c'>0$ so that
$$
s_{2n-1}(z) = \frac{c'}{\sqrt{2(1-\real \alpha_{2n-1})}}  \bigl[ i \rho_{2n-1} \chi_{2n}(z) - i (1-\alpha_{2n-1}) \chi_{2n-1} \bigr] + \cdots .
$$
In fact, $c'=1$ and the ellipsis terms vanish.  To see this, we observe that $s_{2n-1}$, $\chi_{2n}$, and $\chi_{2n-1}$ are all unit vectors and are all perpendicular to elements
in the span of $\{z^{n-1},\ldots,z^{1-n}\}$.  This confirms all odd-indexed columns of $\calN^{-1}=\calN^T$.

To verify the even-index columns (after the zeroth, which is trivial) we first observe that there are some $\gamma >0$ and $\delta\in\C$ so that
$$
s_{2n}(z) = \delta \tfrac{z^n-z^{-n}}{2i} + \cdots +  \gamma \tfrac{z^n+z^{-n}}{2}.
$$
Noting also that
\begin{align*}
& (1-\bar\alpha_{2n-1}) \chi_{2n}(z) + \rho_{2n-1} \chi_{2n-1}(z) \\
{}={} & 2 (1 - \real \alpha_{2n-1}) \tfrac{z^n+z^{-n}}{2} + \cdots + 2(\imag \alpha_{2n-1}) \tfrac{z^n-z^{-n}}{2i}
\end{align*}
we deduce that
$$
s_{2n}(z) = \frac{\gamma'}{\sqrt{2(1-\real \alpha_{2n-1})}}  \bigl[ (1-\bar\alpha_{2n-1}) \chi_{2n}(z) + \rho_{2n-1} \chi_{2n-1}(z)\bigr] + \delta' s_{2n-1}(z)  + \cdots
$$
for some $\gamma'>0$ and $\delta'\in\C$.  Again the ellipsis terms vanish because they are perpendicular everything else that appears in this formula.  To determine $\gamma'$ and $\delta'$, we first exploit that $s_{2n}$ is perpendicular to $s_{2n-1}$.  Using the expansion of $s_{2n-1}$ in terms of $\{\chi_k\}$ proved above, this shows that $\delta'=0$.
Lastly, the fact that $s_{2n}$ is a unit vector confirms $\gamma'=1$.  This justifies the even-indexed columns of $\calN^{-1}=\calN^T$ and so completes the proof of \eqref{SCMV1}.

The symmetry and $7$-diagonal properties of $\calS$ are obvious from the construction.

Finally, suppose $U^T=U$. Note that columns of $R$ are the vectors obtained from the orthonormalization of
$e_1, [\imag U ] e_1, [\real U ]e_1, [\imag (U^2)] e_1, [\real(U^2)] e_1, \ldots$
 But $U^\dagger = \bar{U}$ (entrywise complex conjugation), which means that each of $\real (U^m)$ and $\imag (U^m)$ are real matrices for any $m$. This implies that $R$ is also real.
\end{proof}

\begin{remarks}
1. Similar arguments show that, unless $n=2k$ and $\det U= -1$ (which is equivalent to $\alpha_{n-1}=1$), an $n\times n$ unitary matrix $U$ with $e_1$ cyclic can be reduced to the ``symmetric CMV'' form $$
\calS(\alpha_0,\ldots,\alpha_{n-1}) = \calN \calL \calN^T,
$$
where $\calN = \operatorname{diag}\left( [1],\Psi_1,\Psi_3,\ldots,\Psi_{n-2}\right)$, $\calL = \operatorname{diag}\left( \Xi_0,\Xi_2,\Xi_4,\ldots,\Xi_{n-1}\right)$ if $n=2k+1$ and
$\calN = \operatorname{diag}\left( [1],\Psi_1,\Psi_3,\ldots,\Psi_{n-1}\right)$, $\calL = \operatorname{diag}\left( \Xi_0,\Xi_2,\Xi_4,\ldots,\Xi_{n-2}\right)$ if $n=2k$.
As before, $\Xi_{n-1}=[\bar\alpha_{n-1}]$, while $\Psi_{n-1} = \left[ \sqrt{\alpha_{n-1}} \right]$, where the square root corresponds to the argument $\operatorname{Arg}_{[0,2\pi)}$ with the branch cut along $\bbR_+$ (this is just the $(1,1)$-entry of~\eqref{SCMV3} using $|\alpha_{n-1}|=1$).

2. Similarly, unless $n=2k$ and $\det U = 1$, then one can instead work out an alternative symmetric CMV form $\widetilde{\calS}$ form obtained by orthonormalizing the sequence $\{e_1, [\real U ]e_1, [\imag U ] e_1, [\real(U^2)] e_1, [\imag (U^2)] e_1,  \ldots\}$.
 In this case we obtain the matrix
\begin{equation*}\label{tSCMV1}
\widetilde{\calS}(\alpha_0,\alpha_1,\ldots):= \widetilde{\calN} \calL \widetilde{\calN}^T,
\end{equation*}
 where $\calL$ is as in~\eqref{CMV2}, and
\begin{align*}
\widetilde{\calN} &= \operatorname{diag}\left( [1],\widetilde{\Psi}_1,\widetilde{\Psi}_3,\ldots\right), \\
\widetilde{\Psi}_k & =
\frac{1}{\sqrt{2(1+\real \alpha_k)}}
\left[
\begin{array}{cc}
1+ \bar{\alpha}_k & \rho_k \\
i\rho_k & -i(1+\alpha_k)
\end{array}
\right].
\end{align*}
 \end{remarks}

Note that $\det U =-1$ for $COE(n)$ is a probability 0 event. Therefore with probability $1$ any matrix from $COE(n)$ can be orthogonally reduced to the symmetric CMV form, while fixing $e_1$ and preserving the spectral measure. Moreover, matrices with different spectral measures lead to different sets of Verblunsky coefficients, and therefore different symmetric CMV forms.


The infinite unitary CMV matrix with $\alpha_n=0$ is the bilateral shift (after reordering basis elements).  The symmetric CMV form of this operator is
\begin{equation*}\label{bilateral}
\begin{bmatrix}
0 & \tfrac{i}{\sqrt2} & \tfrac{1}{\sqrt2} & 0 & 0 & 0 & 0 & \\
\tfrac{i}{\sqrt2} & 0 & 0 & \tfrac12 & -\tfrac{i}{2} & 0 & 0 & \\
\tfrac{1}{\sqrt2} & 0 & 0 & \tfrac{i}{2} & \tfrac12 & 0 & 0 & \\
0 & \tfrac12 & \tfrac{i}{2} & 0 & 0 & \tfrac12 & -\tfrac{i}{2} & \\
0 & -\tfrac{i}{2} & \tfrac12 & 0 & 0 & \tfrac{i}{2} & \tfrac12 & \\
0 & 0 & 0 & \tfrac12 & \tfrac{i}{2} & 0 & 0 & \ddots \\
0 & 0 & 0 & -\tfrac{i}{2} & \tfrac12 & 0 & 0 & \\
 &    &   &               &           & \ddots & & \ddots
\end{bmatrix}
\end{equation*}


\smallskip

Similar arguments lead to the analogue of Proposition~\ref{thmSCMV} for $CSE(n)$. The corresponding (complex quaternionic) symmetric CMV form is just $\calS \otimes I_2$ (with $\calS$ as above), and the reducing matrix $R$ is in $\bbUSp(n)$.

\begin{appendix}

\section{Quaternions}\label{sQuaternions}

Let us carefully define the algebras of real quaternions $\bbHR$ and complex quaternions $\bbHC$, as well as discuss quaternionic matrices.

\subsection{Real and complex quaternions}\label{ssQuaternions}

In modern language, the real quaternions were introduced by Hamilton as the unital associative algebra $\bbHR$ over $\R$ generated by $1,\mathsf{i,j,k}$ with (non-commutative) multiplication defined by
\begin{equation*}\label{quaternionMult}
\mathsf{i}^2=\mathsf{j}^2=\mathsf{k}^2=\mathsf{ijk}=-1.
\end{equation*}

We identify the underlying vector space with $\R^4$ by using $1,\mathsf{i},\mathsf{k},\mathsf{j}$ as the standard basis.  Note the deliberate reordering!  This makes for a prettier regular representation $q\mapsto \mathfrak{L}(q)$:  let $q=a+b\mathsf{i}+c\mathsf{k}+d\mathsf{j} \in\bbHR$, then
$$
q (w+x\mathsf{i}+y\mathsf{k}+z\mathsf{j}) = (w'+x'\mathsf{i}+y'\mathsf{k}+z'\mathsf{j})
$$
if and only if
\begin{equation}\label{E:left_times}
\begin{bmatrix} w'\\x'\\y'\\z'\end{bmatrix} = \mathfrak{L}(q) \begin{bmatrix} w\\x\\y\\z\end{bmatrix}
\qtq{with} \mathfrak{L}(q) := \begin{bmatrix} \;a & -b & -c & -d\; \\ \;b & \hphantom{-}a & \hphantom{-}d & -c\; \\
                   \;c & -d & \hphantom{-}a & \hphantom{-}b\; \\ \;d & \hphantom{-}c & -b & \hphantom{-}a\; \end{bmatrix}.
\end{equation}

The $2\times2$ blocks appearing in \eqref{E:left_times} correspond to the standard matrix representations of complex numbers.   This leads to the usual
representation of real quaternions as $2\times 2$ complex matrices:
\begin{equation}\label{E:RSU(2)}
 a+b\mathsf{i}+c\mathsf{k}+d\mathsf{j} \mapsto \begin{bmatrix} a +i b  & -c+id  \\ c+id & \hphantom{-} a-ib  \end{bmatrix} =: \mathfrak{C}(a+b\mathsf{i}+c\mathsf{k}+d\mathsf{j}),
\end{equation}
where $a,b,c,d \in \bbR$.

In particular, this shows that the group of unimodular real quaternions (i.e., $q\in\bbHR$ with $a^2+b^2+c^2+d^2=1$) is the special unitary group $\mathbb{SU}(2)$.


The analogue of \eqref{E:left_times} for multiplication on the right is as follows:
\begin{gather*}\label{E:right_times}
\mathfrak{R}(a+b\mathsf{i}+c\mathsf{k}+d\mathsf{j}) := \begin{bmatrix} \;a & -b & -c & -d\; \\ \;b & \hphantom{-}a & -d & \hphantom{-}c\; \\
                   \;c & \hphantom{-}d & \hphantom{-}a & - b\; \\ \;d & -c & \hphantom{-}b & \hphantom{-}a\; \end{bmatrix},
\end{gather*}
which is to say
$$
(w+x\mathsf{i}+y\mathsf{k}+z\mathsf{j}) q = (w'+x'\mathsf{i}+y'\mathsf{k}+z'\mathsf{j})
\iff\quad\begin{bmatrix} w'\\x'\\y'\\z'\end{bmatrix} = \mathfrak{R}(q) \begin{bmatrix} w\\x\\y\\z\end{bmatrix}.
$$
Note that $q\mapsto\mathfrak{R}(q)$ is not a representation (i.e. homomorphism) of $\HR$, but rather of the reversed algebra.  Needless to say, all matrices $\mathfrak{R}(q)$, $q\in\HR$ commute with all $\mathfrak{L}(p)$, $p\in\HR$.  Indeed, the general theory of matrix algebras (cf. \cite[Theorem~3.4.A]{bWeyl}) guarantees that \emph{every} matrix that commutes with all $\mathfrak{L}(p)$, $p\in\HR$, must be of the form $\mathfrak{R}(q)$ for some $q\in\HR$.

\bigskip

The complex quaternions $\bbHC$ are defined in the same exact way as $\bbHR$ but over $\bbC$ instead of $\bbR$. For $q = a + b \ii + c \kk + d\jj \in\bbHC$ (now $a,b,c,d\in\bbC$), we will be using  the representation $\fC(q)$ defined by same expression~\eqref{E:RSU(2)}.

\bigskip

Note that $\fC(\bbHR)$ consists of $2\times2$ complex matrices $(\alpha_{ij})_{i,j=1}^2$ with $\alpha_{11} = \bar{\alpha}_{22}$ and $\alpha_{12} = -\bar{\alpha}_{21}$, and $\fC(\bbHC)$ consists of all $2\times2$ complex matrices.

For a (real or complex) quaternion $q = a + b \ii + c \kk + d\jj $ let us define its \textit{conjugate} by
$$
\bar{q} = a - b \ii - c \kk - d\jj,
$$
and its \textit{Hermitian conjugate} by
$$
q^\dagger = \bar{a} - \bar{b} \ii - \bar{c} \kk - \bar{d}\jj.
$$
Of course $\bar{q} = q^\dagger$ for $q\in\bbHR$.

\subsection{Matrices of quaternions}\label{ssQuaternionMatrices}

For the rest of the section, let $Q$ be an  $n\times n$ matrix of (real or complex) quaternions 
$Q = [q_{ij}]_{i,j=1}^n$.

We define $\mathfrak{C}(Q)$ 
 as the
matrix formed by replacing each quaternion entry by the corresponding matrix representation
\begin{equation*}\label{fC}
\fC(Q) = \left[\fC(q_{ij})\right]_{i,j=1}^n.
\end{equation*}

Let us define 
$$
Q^\dagger = [q^\dagger_{ji}]_{i,j=1}^n.
$$
It is easy to see that
$$
\fC(Q^\dagger) = \fC(Q)^\dagger,
$$
the usual Hermitian conjugation of complex matrices.

For any $2n \times 2n$ complex matrix $M$ we define its (time reversal) dual matrix by
\begin{equation*}
M^R:= Z^T M^T Z,
\end{equation*}
where
\begin{equation}\label{z}
Z:=I_n \otimes \begin{bmatrix} 0 & -1 \\ 1 & 0\end{bmatrix},
\end{equation}
the $2n\times 2n$ block diagonal matrix with $n$ copies $\begin{bmatrix} 0 & -1 \\ 1 & 0\end{bmatrix}$ along the diagonal.

For a quaternionic matrix $Q$ let us define $Q^R = \fC^{-1}(\fC(Q)^R)$.
It is easy to verify that if $Q$ is complex quaternionic then
$$
Q^R  = [\bar{q}_{ji}]_{i,j=1}^n.
$$
If $Q$ is real quaternionic matrix, then $Q^R=Q^\dagger$ coincide, which implies
\begin{equation}\label{realQuatDuality}
Z^T \fC(Q)^T Z = \fC(Q)^\dagger.
\end{equation}

We say that a  (complex or real) quaternionic matrix is self-dual if $Q^R=Q$.

We say that a (complex or real) quaternionic matrix $Q$ is unitary if $\fC(Q)$ is unitary. 

Finally, by the eigenvalues of a quaternionic matrix $Q$ we mean the eigenvalues of the complex matrix $\fC(Q)$.

\section{Basics of Orthogonal Polynomials on the Unit Circle}\label{sOPUC}



We will introduce some basics of the theory of orthogonal polynomials on the unit circle. For more details we refer the reader, e.g., to~\cite{OPUC1,OPUC2}.

Suppose we have a positive probability measure $\mu$ on the unit circle $\partial\bbD$ having at least $n$ points in its support. For the purposes of this paper, the typical example one should keep in mind  is $\mu$ being the spectral measure of an $n\times n$ unitary matrix $U$ with distinct eigenvalues (see Section~\ref{sSpectral}).

Then applying Gram-Schmidt procedure to the monomials $\{z^k\}_{k=0}^{n-1}$ with respect to the inner product  $\langle f(z),g(z) \rangle :=\int_{\partial\bbD} \overline{f(e^{i\theta})}{g(e^{i\theta})} d\mu(\theta)$ in $L^2(d\mu)$, one obtains a sequence $\{\Phi_k(z)\}_{k=0}^{n-1}$ of monic polynomials orthogonal with respect to $\mu$
\begin{equation}\label{OPUC}
\langle \Phi_j(z), \Phi_k(z) \rangle =0, \quad \mbox{ if } j\ne k.
\end{equation}

The famous result of Szeg\H{o} is that these orthogonal polynomials satisfy the recurrence relation
\begin{equation}\label{recurrence}
\Phi_{k+1}(z) = z \Phi_k(z)-\bar{\alpha}_k \Phi_k^*(z)
\end{equation}
for some sequence of complex coefficients $\{\alpha_j\}_{j=0}^{n-2}$ from $\bbD$ (referred to as Verblunsky coefficients), where for any polynomial $P(z)$ of degree $k$ we define
\begin{equation*}
P^*(z)=z^k \overline{P(\bar{z}^{-1})}
\end{equation*}
(i.e., if $P(z)=
z^k+\kappa_1 z^{k-1}+\ldots+\kappa_{k-1} z+\kappa_k$, then $P^*(z)=\bar{\kappa}_k z^k + \bar{\kappa}_{k-1} z^{k-1} + \ldots + \bar{\kappa}_1 z + 1)$.

It should be noted that if $\mu$ has exactly $n$ points in its support, then $\{z^j \}_{j=0}^{n-1}$ form a basis, and therefore $z^{n}$ is a linear combination of $\{z^j \}_{j=0}^{n-1}$ in $L^2(d\mu)$. However Gram--Schmidt procedure can still be applied to produce a polynomial $\Phi_n(z)$ (which is of norm $0$ in $L^2(d\mu)$, of course), and~\eqref{recurrence} holds with a unimodular coefficient $\alpha_{n-1}$.

In fact, there is a one-to-one correspondence between all probability measures supported on $n$ distinct points
\begin{equation}\label{finiteMeasure}
\mu(\theta)=\sum_{j=1}^n \mu_j \delta_{e^{i\theta_j}}
\end{equation}
and sequences of Verblunsky coefficients $\{\alpha_k\}_{k=0}^{n-1}$ with $\alpha_k \in \bbD$ for $0\le k \le n-2$ and $\alpha_{n-1}\in\partial\bbD$. So $\alpha_0,\ldots,\alpha_{n-1}$ can be viewed as a change of variables from $e^{i\theta_1},\ldots,e^{i\theta_n},\mu_1,\ldots,\mu_n$ (subject to $\sum_{j=1}^n \mu_j =1$). This was initially explored in~\cite{KN} (see also Forrester--Rains~\cite{ForRai06}), and is developed further in our Section~\ref{sModels}.



For each $0\le k \le n-1$, denote the orthonormal polynomials
\begin{equation*}
\phi_k(z)=\tfrac{1}{||\Phi_k||}\Phi_k(z),
\end{equation*}
where the norm is in $L^2(d\mu)$. We summarize some of the properties of orthogonal polynomials that we will need in the next lemma.


\begin{lemma}\label{OP_lemma}
\begin{itemize}
\item[(i)]  Let $1\le k \le n$ and $\Phi_k(z)=\prod_{j=1}^k (z-z_j)$. Then
\begin{equation*}
\Phi_k(0)  = (-1)^k \prod_{j=1}^k z_j = -\bar\alpha_{k-1} ;
\end{equation*}
If all $\alpha_j$'s are real, then
\begin{align*}
\Phi_k(1) & = \prod_{j=1}^k (1-z_j) = \prod_{j=0}^{k-1} (1-\alpha_j) ;\\
\Phi_k(-1) &=  \prod_{j=1}^n (1+z_j) = \prod_{j=0}^{k-1} (1+(-1)^j\alpha_j ).
\end{align*}
\item[(ii)]  For any $0\le k \le n$, $||\Phi_k||^2 = \prod_{j=0}^{k-1} (1-|\alpha_j|^2)$.
\item[(iii)] For any  $0\le k \le n-1$, the Christoffel--Darboux formula holds:
\begin{equation*}
\sum_{j=0}^{k-1} \phi_j(z) \overline{\phi_j(\zeta)} = \frac{\phi_k^*(z) \overline{\phi_k^*(\zeta)} - \phi_k(z) \overline{\phi_k(\zeta)}}{1-z\bar{\zeta}}.
\end{equation*}
\item[(iv)] The Cauchy formula holds:
\begin{equation*}\label{cauchy}
\det \left[ \frac{1}{1-z_j \bar{z}_s} \right]_{1\le j,s \le k} = \frac{\prod_{j<s}|z_s-z_j|^2}{\prod_{j,s=1}^k(1-z_j \bar{z}_s)}.
\end{equation*}
\item[(v)] If $\mu$ is of the form~\eqref{finiteMeasure}, then
$$
\prod_{j=0}^{n-2} (1-|\alpha_j|^2)^{n-j-1} = \left|\Delta\left(e^{i\theta_1},\ldots,e^{i\theta_n}\right)\right|^2 \prod_{j=1}^n \mu_j .
$$
\item[(vi)] For any  $0\le k \le n$, if $\Phi_k(z)=\prod_{j=1}^k (z-z_j)$ then
\begin{equation}\label{z-alpha}
\prod_{j=0}^{k-1} (1-|\alpha_j|^2)^{j+1} = \prod_{j,s=1}^k (1-z_j \bar{z}_s).
\end{equation}
\end{itemize}
\end{lemma}
\begin{proof}
(i) is immediate from~\eqref{recurrence}. (ii) and (iii) are standard facts, see~\cite{OPUC1,OPUC2}. A proof of (iv) can be found in, e.g., \cite[Lemma~7.6.A]{bWeyl}. (v) is proved in~\cite[Lem~4.1]{KN}.

Let us now prove (vi).  An alternate proof of this in \cite[Thm 2.1.4]{OPUC1}.  By elementary row operations,
\begin{equation*}
\begin{aligned}
|\Delta(z_1,\ldots,z_k)|^2  &= \det\left( \left[z_j^{s-1}\right]_{1\le j,s \le k}
\left[\bar{z}_s^{j-1}\right]_{1\le j,s \le k} \right) \\
&= \det \left( \left[\Phi_{s-1}(z_j)\right]_{1\le j,s \le k}
\left[\overline{\Phi_{j-1}(z_s)}\right]_{1\le j,s \le k} \right) \\
&= \left(\prod_{s=0}^{k-1} ||\Phi_s||^2 \right) \det \left( \left[\phi_{s-1}(z_j)\right]_{1\le j,s \le k}
\left[\overline{\phi_{j-1}(z_s)}\right]_{1\le j,s \le k} \right).
\end{aligned}
\end{equation*}
Note that $\phi_k^*(z)=\tfrac{1}{||\Phi_k||}\prod_{j=1}^k (1-z \bar{z}_j)$ and $\phi_k(z_s)=0$ for any $s$. Therefore the right-hand side of the last equation can be written as
\begin{equation*}
\begin{aligned}
&= \left(\prod_{s=0}^{k-1} ||\Phi_s||^2 \right) \det \left[ \frac{\phi_k^*(z_j) \overline{\phi_k^*(z_s)} }{1-z_j \bar{z}_s} \right]_{1\le j,s \le k} \\
&= \left(\prod_{s=0}^{k-1} ||\Phi_s||^2 \right) \left(\prod_{s=1}^{k} |\phi_k^*(z_s)|^2 \right) \det \left[ \frac{1}{1-z_j \bar{z}_s} \right]_{1\le j,s \le k}.
\end{aligned}
\end{equation*}
Now using parts (ii), and (iv), we obtain
\begin{equation*}
=\frac{|\Delta(z_1,\ldots,z_k)|^2\, \prod_{j,s=1}^k |1-z_j \bar{z}_s|^2}{\left(\prod_{s=0}^{k-1} (1-|\alpha_j|^2)^{j+1} \right) \prod_{j,s=1}^k(1-z_j \bar{z}_s)  },
\end{equation*}
which gives us \eqref{z-alpha} after cancelling by $|\Delta|^2$.
\end{proof}

%

\section{Some Distributions on Spheres}\label{sSphere}
For the reader's convenience we collect here some of the distributions needed in the body of the text. 

\begin{lemma}\label{lemSphere}
(a) Let a random vector $x\in \bbR^n$ be chosen according to the normalized uniform distribution the unit sphere $\{x\in\bbR^n : ||x||=1\}$. Then
$$
(\mu_1,\ldots,\mu_n) = (x_1^2,\ldots,x_n^2)
$$
are jointly distributed on the simplex
$
\sum_{j=1}^{n} \mu_j =1, \mu_j \ge0 , 1\le j \le n,
$
according to the probability distribution
$$
\tfrac{\Gamma\left(n/2\right)}{\left[\Gamma\left(1/2\right)\right]^n} \prod_{j=1}^n \mu_j^{-1/2} d\mu_1\ldots d\mu_{n-1}.
$$

(b) Let a random vector $x\in \bbR^{2n}$ be chosen according to the normalized uniform distribution the unit sphere $\{x\in\bbR^{2n} : ||x||=1\}$. Then
$$
(\mu_1,\ldots,\mu_{n}) = (x_1^2+x_2^2,\ldots,x_{2n-1}^2+x_{2n}^2)
$$
are jointly distributed on the simplex
$
\sum_{j=1}^{n} \mu_j =1, \mu_j \ge0 , 1\le j \le n,
$
according to the probability distribution
$$
(n-1)! \, d\mu_1\ldots d\mu_{n-1}.
$$

(c) Let a random vector $x\in \bbR^{2n}$ be chosen according to the normalized uniform distribution the unit sphere $\{x\in\bbR^{2n} : ||x||=1\}$. Then
$$
(\mu_1,\mu_2,\ldots,\mu_{n-1},\mu_{n},\mu_{n+1}) = (x_1^2+x_2^2,x_3^2+x_4^2,\ldots,x_{2n-3}^2+x_{2n-2}^2,x_{2n-1}^2,x_{2n}^2)
$$
are jointly distributed on the simplex
$
\sum_{j=1}^{n+1} \mu_j =1, \mu_j \ge0 , 1\le j \le n+1,
$
according to the probability distribution
$$
\tfrac{(n-1)!}{\pi}
  \tfrac{1}{\sqrt{\mu_n \mu_{n+1}}}  d\mu_1\ldots d\mu_n.
$$

(d) Let a random vector $x\in \bbR^{2n+1}$ be chosen according to the normalized uniform distribution the unit sphere $\{x\in\bbR^{2n+1} : ||x||=1\}$. Then
$$
(\mu_1,\ldots,\mu_{n},\mu_{n+1}) = (x_1^2+x_2^2,x_3^2+x_4^2,\ldots,x_{2n-1}^2+x_{2n}^2,x_{2n+1}^2)
$$
are jointly distributed on the simplex
$
\sum_{j=1}^{n+1} \mu_j =1, \mu_j \ge0 , 1\le j \le n+1,
$
according to the probability distribution
$$
\tfrac{\Gamma(n+\tfrac{1}{2})}{\Gamma(\tfrac12)}
 \tfrac{1}{\sqrt{\mu_{n+1}}}  d\mu_1\ldots d\mu_n.
$$

(e) Let a random vector $x\in \bbR^{4n}$ be chosen according to the normalized uniform distribution the unit sphere $\{x\in\bbR^{4n} : ||x||=1\}$. Then
$$
(\mu_1,\ldots,\mu_{n}) = (x_1^2+x_2^2+x_3^2+x_4^2,\ldots,x_{4n-3}^2+x_{4n-2}^2+x_{4n-1}^2+x_{4n}^2)
$$
are jointly distributed on the simplex
$
\sum_{j=1}^{n} \mu_j =1, \mu_j \ge0 , 1\le j \le n,
$
according to the probability distribution
$$
(2n-1)!
 \prod_{j=1}^n \mu_j  d\mu_1\ldots d\mu_{n-1}.
$$
\end{lemma}
\begin{proof}
Part (b) has been proved in~\cite[Cor~A.3]{KN}.

Part (a) follows from~\cite[Cor~A.1]{KN} and a change of variables. The normalization constant comes from~\cite[Cor~A.4]{KN}.

Parts (c), (d), (e) can be proved using the arguments in the proof of~\cite[Cor~A.3]{KN} and then using ~\cite[Cor~A.4]{KN} for the normalization constants.
\end{proof}

%
%

\section{Jacobian Determinants}\label{sJacobians}
Let us define the Jacobian determinants of real and complex maps. 

Recall that for a map $f:\bbR^n \to \bbR^n$,  $(z_1,\ldots,z_n) \mapsto (f_1,\ldots,f_n)$, the Jacobian determinant is defined to be
$$
J\left[\frac{\partial(f_1,\ldots,f_n)}{\partial(z_1,\ldots,z_n)}\right]:=
\det \left[ \frac{\partial f_j}{\partial z_k} \right]_{j,k=1}^n.
$$

To avoid confusion, for complex maps we will use a different notation. If each component of the map $f:\bbC^n\to \bbC^n$, $(z_1,\ldots,z_n) \mapsto (f_1,\ldots,f_n)$ is analytic, then one can define the complex Jacobian determinant
$$
J_\bbC\left[\frac{\partial(f_1,\ldots,f_n)}{\partial(z_1,\ldots,z_n)}\right]:=\det \left[ \frac{\partial f_j}{\partial z_k} \right]_{j,k=1}^n.
$$

For our purposes however, the following will be more natural. For a (not necessarily analytic) map $f:\bbC^n\to \bbC^n$, $(z_1,\ldots,z_n) \mapsto (f_1,\ldots,f_n)$, the real Jacobian determinant $J_\bbR$ is defined to be
$$
J_\bbR\left[\frac{\partial(f_1,\ldots,f_n)}{\partial(z_1,\ldots,z_n)}\right]:= J\left[\frac{\partial(\real f_1, \imag f_1,\ldots,\real f_n,\imag f_n)}{\partial(\real z_1, \imag z_1,\ldots,\real z_n,\imag z_n)}\right].
$$

We show below that for a given map $f:\bbC^n\to \bbC^n$, $(z_1,\ldots,z_n) \mapsto (f_1,\ldots,f_n)$, its Jacobian $J_\bbR f$ is equal to the Jacobian $J_\bbC \widetilde{f}$, where
$
\widetilde{f}:\bbC^{2n} \to \bbC^{2n}
$
is the mapping that takes $(z_1,\bar{z}_1,\ldots,z_n,\bar{z}_n)$ into $(f_1,\bar{f}_1,\ldots,f_n,\bar{f}_n)$, where we treat $z_j$ and $\bar{z}_j$ as independent variables.

\begin{lemma}\label{lemJacobian}
For a map $f:\bbC^n\to \bbC^n$, $(z_1,\ldots,z_n) \mapsto (f_1,\ldots,f_n)$,
$$
J_\bbR\left[\frac{\partial(f_1,\ldots,f_n)}{\partial(z_1,\ldots,z_n)}\right] = 
J_\bbC\left[\frac{\partial(f_1,\bar{f}_1,\ldots,f_n,\bar{f}_n)}{\partial(z_1,\bar{z}_1,\ldots,z_n,\bar{z}_n)}\right].
$$
\end{lemma}
\begin{remark}
This implies that for an analytic map $f:\bbC^n\to \bbC^n$,
$$
J_\bbR\left[\frac{\partial(f_1,\ldots,f_n)}{\partial(z_1,\ldots,z_n)}\right]= \left|J_\bbC\left[\frac{\partial(f_1,\ldots,f_n)}{\partial(z_1,\ldots,z_n)}\right] \right|^2.
$$
\end{remark}
\begin{proof}

Indeed,
\begin{multline*}\left[\begin{array}{cc}
\frac{\partial f_j}{\partial z_k} & \frac{\partial \bar{f}_j}{\partial z_k} \\
\frac{\partial f_j}{\partial \bar{z}_k} & \frac{\partial \bar{f}_j}{\partial \bar{z}_k}
\end{array}
\right]_{1\le j,k\le n}
\\
=
\left(\operatorname{diag}\left[
\begin{array}{cc}
1/2 & -i/2 \\
i/2 & i/2
\end{array}
\right] \right)
\left[
\begin{array}{cc}
\tfrac{\partial \real f_j}{\partial \real z_k} & \tfrac{\partial \imag f_j}{\partial \real z_k} \\
\tfrac{\partial \real f_j}{\partial \imag z_k} & \tfrac{\partial \imag f_j}{\partial \imag z_k}
\end{array}
\right]_{1\le j,k\le n}
\left(\operatorname{diag} \left[
\begin{array}{cc}
1 & 1 \\
i & -i
\end{array}
\right]\right),
\end{multline*}
where $\operatorname{diag}[A]$ denotes a block diagonal matrix with $n$ copies of $A$ on the diagonal.
\end{proof}

In the following lemma we let
$$
\Phi(z) = z^n + \kappa_1 z^{n-1} + \ldots + \kappa_{n-1} z +\kappa_n = \prod_{j=1}^n (z-z_j),
$$
where we order $z_j$'s, e.g., by the absolute values and then by the argument. The mapping from the zeros to coefficients  is then injective. It can be viewed as $\bbR^n\to\bbR^n$ or as $\bbC^n\to\bbC^n$.

\begin{lemma}\label{jacobian}
\begin{itemize}
  \item[(a)] As an $\bbR^n\to\bbR^n$ map,
  $$
  J\left[\frac{\partial(\kappa_1,\ldots,\kappa_n)}{\partial(z_1,\ldots,z_n)}\right]=\Delta(z_1,\ldots,z_n);
  $$
  \item[(b)] As a $\bbC^n\to\bbC^n$ map,
  \begin{equation*}
  J_\bbR\left[\frac{\partial(\kappa_1,\ldots,\kappa_n)}{\partial(z_1,\ldots,z_n)}\right] = | \Delta(z_1,\ldots,z_n)|^2.
  \end{equation*}
\end{itemize}
\end{lemma}
\begin{proof}
Let $$s_k=\sum_{j_1<\ldots< j_k} z_{j_1}\ldots z_{j_k}$$
and $$p_k=\sum_{j=1}^n z_j^k.$$
Then by the Newton identity, $k s_k = (-1)^{k-1} p_k + \sum_{j=1}^{k-1} (-1)^{j-1}s_{k-j} p_j$, and so inductively we get that $k s_k = (-1)^{k-1} p_k +$linear combination of $s_1,\ldots,s_{k-1}$. Thus noting that $\kappa_k=(-1)^k s_k$ and using Gaussian elimination, we obtain part (a) of the lemma.

Part (b) then follows by using the remark after Lemma~\ref{lemJacobian} and analyticity of the change of variables.
\end{proof}
\end{appendix}

%
%
%

%

\noindent \textbf{{Acknowledgement}}

\bigskip
\noindent  The authors would like to thank Yan V. Fyodorov for drawing our attention to resonances for the systems with non-ideal coupling.

The first author was funded, in part, by US NSF grant DMS-1265868. The second author was funded, in part, by the grant KAW 2010.0063 from the Knut and Alice Wallenberg Foundation.

\bibliographystyle{plain}
\bibliography{mybib_rmt,mybib}

\end{document}